\documentclass[a4paper,11pt]{amsart}
\usepackage[matrix,arrow,tips,curve]{xy}
\usepackage[english]{babel}
\usepackage{amsmath}
\usepackage{amssymb}
\usepackage{tikz}
\usepackage{mathrsfs}
\usepackage{enumerate}
\usepackage{graphicx}
\usepackage{hyperref}
\usetikzlibrary{trees}
\usetikzlibrary{arrows}

\newtheorem{thm}{Theorem}[section]
\newtheorem{Lemma}[thm]{Lemma}
\newtheorem{Proposition}[thm]{Proposition}
\newtheorem{Corollary}[thm]{Corollary}

\newtheorem*{thm*}{Theorem}

\theoremstyle{definition}

\newtheorem{Definition}[thm]{Definition}
\newtheorem{Remark}[thm]{Remark}

\newtheorem{Example}[thm]{Example}

\newtheorem{say}[thm]{}

\definecolor{wwwwww}{rgb}{0.4,0.4,0.4}

\renewcommand{\P}{\mathbb{P}}

\newcommand{\G}{\mathcal{G}}
\newcommand{\LG}{\mathcal{LG}}
\newcommand{\C}{\mathbb{C}}
\newcommand{\Sp}{Sp}

\DeclareMathOperator{\pf}{pf}

\DeclareMathOperator{\Sec}{Sec}

\DeclareMathOperator{\expdim}{expdim}

\hypersetup{pdfpagemode=UseNone}
\hypersetup{pdfstartview=FitH}
		
\begin{document}
\title[\resizebox{6.1in}{!}{Projective aspects of the Geometry of Lagrangian Grassmannians and Spinor varieties}]{Projective aspects of the Geometry of Lagrangian Grassmannians and Spinor varieties}

\author[Ageu Barbosa Freire]{Ageu Barbosa Freire}
\address{\sc Ageu Barbosa Freire\\
Instituto de Matem\'atica e Estat\'istica, Universidade Federal Fluminense, Campus Gragoat\'a, Rua Alexandre Moura 8 - S\~ao Domingos\\
24210-200 Niter\'oi, Rio de Janeiro\\ Brazil}
\email{ageufreire@id.uff.br}

\author[Alex Massarenti]{Alex Massarenti}
\address{\sc Alex Massarenti\\ Dipartimento di Matematica e Informatica, Universit\`a di Ferrara, Via Machiavelli 30, 44121 Ferrara, Italy\newline
\indent Instituto de Matem\'atica e Estat\'istica, Universidade Federal Fluminense, Campus Gragoat\'a, Rua Alexandre Moura 8 - S\~ao Domingos\\
24210-200 Niter\'oi, Rio de Janeiro\\ Brazil}
\email{alex.massarenti@unife.it, alexmassarenti@id.uff.br}

\author[Rick Rischter]{Rick Rischter}
\address{\sc Rick Rischter\\
Universidade Federal de Itajub\'a\\ 
Av. BPS 1303, Bairro Pinheirinho\\ 
37500-903 Itajub\'a, Minas Gerais\\ 
Brazil}
\email{rischter@unifei.edu.br}

\date{\today}

\begin{abstract}
We study the projective behavior, mainly with respect to osculating spaces and secant varieties, of Lagrangian Grassmannians and Spinor varieties. We prove that these varieties have osculating dimension smaller than expected. Furthermore, we give numerical conditions ensuring the non secant defectivity of Lagrangian Grassmannians in their Pl\"ucker embedding and of Spinor varieties in both their Pl\"ucker and Spinor embeddings. 
\end{abstract}

\maketitle
\tableofcontents

\section{Introduction}
Let $V$ be a vector space endowed with a non-degenerate quadratic form $Q$ or, when $\dim(V)$ is even, with a non-degenerated symplectic form $\omega$. For $r\leq\frac{\dim(V)}{2}$ the isotropic Grassmannians $\mathcal{G}_{Q}(r,V), \mathcal{G}_{\omega}(r,V)$ are the subvarieties of the Grassmannian $\mathcal{G}(r,V)$ parametrizing $r$-dimensional subspaces of $V$ that are isotropic with respect to $Q$ and $\omega$ respectively.  

All isotropic Grassmannians, with the exception of the symmetric case when $\dim(V) = 2n$ is even and $r = n$, are irreducible. Furthermore, in the exceptional case the isotropic Grassmannian $\mathcal{G}_{Q}(n,V)$ has two connected components each one parametrizing the linear subspaces in one of the two families of $(n-1)$-planes of $\mathbb{P}(V)$ contained in the smooth quadric hypersurface in $\mathbb{P}(V)$ defined by $Q$. Either of these two isomorphic components is called the $\frac{n(n-1)}{2}$-dimensional Spinor variety and denoted by $\mathcal{S}_n$.

The restriction of the Pl\"ucker embedding of $\mathcal{G}(n,V)$ induces an embedding $\mathcal{S}_n\rightarrow\mathbb{P}(\bigwedge^nV_{+})$. However, this is not the minimal homogeneous embedding of $\mathcal{S}_n$ that we will denote by $\mathcal{S}_n\rightarrow\mathbb{P}(\Delta)$ and refer to as the Spinor embedding. The Pl\"ucker embedding of $\mathcal{S}_n$ can be obtained by composing the Spinor embedding with the degree two Veronese embedding. 

In the skew-symmetric case, again when $d = 2n$ is even and $r = n$, the isotropic Grassmannian $\mathcal{G}_{\omega}(n,V)$ is called the $\frac{n(n+1)}{2}$-dimensional Lagrangian Grassmannian and denoted by $\mathcal{LG}(n,2n)$. Unlike the case of the Spinor variety, the restricting of the Pl\"ucker embedding of $\mathcal{G}(n,V)$ yields the minimal homogeneous embedding of $\mathcal{LG}(n,2n)$ that we will denote by $\mathcal{LG}(n,2n)\rightarrow\mathbb{P}(V_{\omega_n})$.

Lagrangian Grassmanninas and Spinor varieties have been widely studied both from the geometrical and the representational theoretical viewpoint \cite{LM03}, \cite{Ma09}, \cite{BB11}, \cite{IR05}, \cite{An11}, \cite{SV10}, \cite{Pe12}. In this paper we will focus on the projective geometry of these varieties, mainly on the dimension of their osculating spaces and secant varieties.

Let $X\subset\mathbb{P}^N$ be an irreducible variety of dimension $n$, and $p\in X$ a smooth point. The \textit{$s$-osculating space} $T_p^{s}X$ of $X$ at $p$ is essentially the linear subspace of $\mathbb{P}^N$ generated by the partial derivatives of order less or equal than $s$ of a local parametrization of $X$ at $p$. Note that while the dimension of the tangent space at a smooth point is always equal to the dimension of the variety, the dimension of higher order osculating spaces can be strictly smaller than expected even at a general point. In general, we have $\dim(T_p^s X) = \min\left\{\binom{n+s}{n}-1-\delta_{s,p},N\right\}$, where $\delta_{s,p}$ is the number of independent differential equations of order less or equal than $s$ satisfied by $X$ at $p$. Such dimension is called the general $s$-osculating dimension of $X$. Projective varieties having general $s$-osculating dimension smaller than expected were introduced and studied in \cite{Se07}, \cite{Te12}, \cite{Bo19}, \cite{To29}, \cite{To46}, and more recently in \cite{DDI13}, \cite{DI15}, \cite{DIV14}, \cite{DJL17}, \cite{FI02}, \cite{Il99}, \cite{Il06}, \cite{MMO13}, \cite{PT90}. In Corollaries \ref{dimOscLG}, \ref{dim_Osc_S1}, \ref{dim_osc_S2} we compute the general $s$-osculating dimension of $\mathcal{LG}(n,2n)$ in the Pl\"ucker embedding and of $\mathcal{S}_n$ in both the Pl\"{u}cker and the Spinor embedding. In particular, we have the following result.

\begin{thm}\label{main_1}
The Lagrangian Grassmannian $\mathcal{LG}(n,2n)\subset\mathbb{P}(V_{\omega_n})$ and the Spinor variety $\mathcal{S}_n\subset\mathbb{P}(\bigwedge^nV_{+})$ have $s$-osculating dimension smaller than expected respectively for $2\leq s\leq n-1$ and for $2\leq s\leq 2\lfloor\frac{n}{2}\rfloor-1$. Furthermore, for any $s\geq 0$ we have $T_p^s \mathcal{LG}(n,2n) = T_p^s\mathcal{G}(n,V)\cap \mathbb{P}(V_{\omega_n})$ for any $p\in \mathcal{LG}(n,2n),$ and $T_p^s \mathcal{S}_n = T_p^s\mathcal{G}(n,V)\cap \mathbb{P}(\bigwedge^nV_{+})$ for any $p\in \mathcal{S}_n$.

Finally, the Spinor variety $\mathcal{S}_n\subset\mathbb{P}(\Delta)$ has $s$-osculating dimension smaller than expected for $2\leq s\leq \lfloor\frac{n}{2}\rfloor-1$, and 
$T_p^n \mathcal{S}_n =\mathbb{P}(\Delta)$. 
\end{thm}

In Section \ref{sec_def} we study the dimension of the secant varieties of Lagrangian Grassmannians and Spinor varieties. The \textit{$h$-secant variety} $\mathbb{S}ec_{h}(X)$ of a non-degenerate $n$-dimensional variety $X\subset\mathbb{P}^N$ is the Zariski closure of the union of all linear spaces spanned by collections of $h$ points of $X$. Secant varieties of homogeneous varieties are fundamental objects for instance in tensor decomposition problems. Indeed, they have been used to construct and study moduli spaces for additive decompositions of a general tensor into a given number of rank one tensors \cite{Do04}, \cite{DK93}, \cite{Ma16}, \cite{MM13}, \cite{RS00}, \cite{TZ11}, \cite{BGI11}.

The \textit{expected dimension} of $\mathbb{S}ec_{h}(X)$ is $\expdim(\mathbb{S}ec_{h}(X)):= \min\{nh+h-1,N\}$. The actual dimension of $\mathbb{S}ec_{h}(X)$ may be smaller than the expected one. Following \cite{Za93}, we say that $X$ is \textit{$h$-defective} if $\dim(\mathbb{S}ec_{h}(X)) < \expdim(\mathbb{S}ec_{h}(X))$. Determining secant defectivity is a classical problem in algebraic geometry which goes back to the Italian school \cite[Chapter 10]{Ca37}, \cite{Sc08}, \cite{Se01}, \cite{Te11}. 

We tackle secant defectivity of Lagrangian Grassmannians and Spinor varieties following the strategy introduced in \cite{MR19}, which we now explain. Given general points $x_1,\dots,x_h\in X\subset\mathbb{P}^N$, consider the linear projection $\tau_{X,h}:X\subseteq\mathbb{P}^N\dasharrow\mathbb{P}^{N_h}$, with center $\left\langle T_{x_1}X,\dots,T_{x_h}X\right\rangle$, where $N_h:=N-1-\dim (\left\langle T_{x_1}X,\dots,T_{x_h}X\right\rangle)$.

By \cite[Proposition 3.5]{CC02}, if $\tau_{X,h}$ is generically finite then $X$ is not $(h+1)$-defective. In \cite{MR19} a new strategy  was developed, based on the more general \emph{osculating projections} instead of just tangential projections. Given $p_1,\dots, p_l\in X$ general points, we consider the linear projection $\Pi_{T^{k_1,\dots,k_l}_{p_1,\dots,p_l}}:X\subset\mathbb{P}^N\dasharrow\mathbb{P}^{N_{k_1,\dots,k_l}}$ with center $\left\langle T_{p_1}^{k_1}X,\dots, T_{p_l}^{k_l}X\right\rangle$, and call it a \textit{$(k_1+\dots +k_l)$-osculating projection}, where $N_{k_1,\dots,k_l}:=N-1-\dim (\left\langle T_{p_1}^{k_1}X,\dots, T_{p_l}^{k_l}X\right\rangle)$. Under suitable conditions, one can degenerate the linear span of several tangent spaces $T_{x_i}X$ into a subspace contained in a single osculating space $T_p^{k}X$. 
So the tangential projections $\tau_{X,h}$ degenerate to a linear projection with center contained in the linear span of osculating spaces $\left\langle T_{p_1}^{k_1}X,\dots, T_{p_l}^{k_l}X\right\rangle$. If $\Pi_{T^{k_1,\dots,k_l}_{p_1,\dots,p_l}}$ is generically finite, then $\tau_{X,h}$ is also generically finite, and one concludes that $X$ is not $(h+1)$-defective. The advantage of this approach is that we are allowed  to consider osculating spaces at much less points than $h$, and consequently to control the dimension of the general fiber of the projection. 

This strategy was successfully applied to study the problem of secant defectivity for Grassmannians \cite{MR19} and Segre-Veronese varieties \cite{AMR17}. Here we apply it to Lagrangian Grassmannians and Spinor varieties. While for the Pl\"ucker embeddings, thanks to the relation among the osculating space of these varieties and those of $\mathcal{G}(n,V)$ in Theorem \ref{main_1}, our arguments boil down to the main results on the osculating behavior of $\mathcal{G}(n,V)$ in \cite{MR19}, for $\mathcal{S}_n$ in its Spinor embedding more complicated computations are needed.  

At the best of our knowledge very few is know about secant defectivity of these varieties \cite{BB11}, \cite{An11}. In particular, it has been conjectured that $\mathbb{S}ec_h(\mathcal{LG}(n,2n))$ has the expected dimension except for the cases $(n,h)\in\{(4,3),(4,4)\}$ \cite[Conjecture 1.2]{BB11}. The main results in Theorems \ref{Main_LG}, \ref{main_SP}, \ref{th_main_SS} can be summarized as follows.

\begin{thm}\label{main_2}
The Lagrangian Grassmannian $\mathcal{LG}(n,2n)\subset\mathbb{P}(V_{\omega})$ in its Pl\"ucker embedding is not $h$-defective for $h\leq\left\lfloor\frac{n+1}{2}\right\rfloor$. Furthermore, the Spinor variety $\mathcal{S}_n\subset\mathbb{P}(\bigwedge^nV_{+})$ in its Pl\"ucker embedding is not $h$-defective for $h\leq\left\lfloor\frac{n}{2}\right\rfloor$.

Finally, the Spinor variety $\mathcal{S}_n\subset\mathbb{P}(\Delta)$ in its Spinor embedding is not $h$-defective for $h\leq\left\lfloor\frac{n+2}{4}\right\rfloor$.
\end{thm}

In Section \ref{sec_def} we observe that Theorem \ref{main_2} improves the main results on secant defectivity of $\mathcal{LG}(n,2n)\subset\mathbb{P}(V_{\omega})$ in \cite{BB11} for $n\geq 9$, and the main results on secant defectivity of $\mathcal{S}_n\subset\mathbb{P}(\Delta)$ in \cite{An11} for $n\geq 14$. 

The paper is organized as follows. In Section \ref{IsG} we recall some basic facts on Lagrangian Grassmannians and Spinor varieties describing the local parametrizations that are needed in the following sections. In Section \ref{osculating} we study higher osculating spaces of these varieties and the linear projections centered at them. Finally, in Section \ref{sec_def} we prove our main results on secant defectivity of Lagrangian Grassmannians and Spinor varieties. 

\subsection*{Acknowledgments}
The first named author would like to thank FAPERJ for the financial support. 
The second named author is a member of the Gruppo Nazionale per le Strutture Algebriche, Geometriche e le loro Applicazioni of the Istituto Nazionale di Alta Matematica "F. Severi" (GNSAGA-INDAM).

\section{Lagrangian Grassmannians and Spinor varieties}\label{IsG}
Throughout the paper we work over the field of complex numbers. Let $V$ be a complex vector space of dimension $d.$ We will denote by $\G(r,V)$ the Grassmannian parametrizing vector subspaces of dimension $r$ of $V$, in its Pl\"ucker embedding, that is the morphism induced by the determinant of the universal quotient bundle $\mathcal{Q}_{\G(r,V)}$ on $\G(r,V)$:
$$
\begin{array}{cccc}
f_{r,d}: &\G(r,V)& \longrightarrow & \P^N:=\P(\bigwedge^{r}V)\\
      & \left\langle v_1,\dots,v_r\right\rangle & \mapsto & [v_1\wedge \dots\wedge v_r]
\end{array}
$$
where $N = \binom{d}{r}-1$. Now, let $\Lambda:=\left\{ I\subset \{1,\dots,d\}, |I|=r \right\}$. For each $I=\{i_1,\dots,i_r\}\in \Lambda$, with $i_1<\dots < i_r$, let $e_I\in\G(r,V)$ be the point corresponding to $e_{i_1}\wedge\dots\wedge e_{i_r}\in\bigwedge^{r}V$. We will denote by $p_I$ the Pl\"ucker coordinates on $\mathbb{P}^N$.

Let $Q$ be a non-degenerate quadratic form on $V$. A subspace $W\subseteq V$ is isotropic with respect to $Q$ if $Q(v_1,v_2) = 0$ for any $v_1,v_2\in W$. For any $r\leq\frac{d}{2}$ there exists a $SO(V)$-equivariant projective variety $\mathcal{G}_{Q}(r,V)$ parametrizing $r$-dimensional subspaces of $V$ that are isotropic with respect to $Q$. 

Similarly, if $V$ is endowed with a non-degenerate symplectic form $\omega$, for any $r\leq\frac{d}{2}$ there exists a $Sp(V)$-equivariant projective variety $\mathcal{G}_{\omega}(r,V)$ parametrizing $r$-dimensional subspaces of $V$ that are isotropic with respect to $\omega$. Note that in this case $d$ must be even. 

The varieties $\mathcal{G}_{Q}(r,V),\mathcal{G}_{\omega}(r,V)$ are called \textit{isotropic Grassmannians}. All these varieties, except in the symmetric case when $d = 2r$, are irreducible of Picard rank one \cite[Section 2.1]{Te05}. The exceptional case $\mathcal{G}_{Q}(r,V)$ with $d = 2r$ has two irreducible components. 

In the skew-symmetric case $\bigwedge^rV$ is a reducible $Sp(V)$ representation. Indeed, it contains the irreducible submodule $\omega\wedge\bigwedge^{r-2}V\subset\bigwedge^{r}V$. The complementary submodule $V_{\omega}$ is irreducible, and the restriction of the Pl\"ucker embedding induces and embedding
$$\mathcal{G}_{\omega}(r,V)\rightarrow \mathbb{P}(V_{\omega})\subseteq \mathbb{P}(\bigwedge^rV)$$

In the symmetric case, if $d = 2m+1$ and $r < m$ or $d = 2m$ and $r<m-1$, then $\bigwedge^r V$ is an irreducible $SO(V)$-module. Then, the image of $\mathcal{G}_{Q}(r,V)$ under the Pl\"ucker embedding spans the whole of $\mathbb{P}(\bigwedge^rV)$.  

In this paper we will study two classes of isotropic Grassmannians. The \textit{Lagrangian Grassmannian} $\mathcal{LG}(n,2n)$ is the subvariety of $\G(n,V)$ parametrizing dimension $n$ Lagrangian subspaces of a complex symplectic vector space $V$ of dimension $2n$. We fix a basis of $V$ such that the symplectic form in this basis is given by
$$
J = \left(\begin{array}{cc}
0 & I_{n}\\ 
-I_{n} & 0
\end{array}\right) 
$$
Let us represent a point $W\in \LG(n,2n)$ as a $2n\times n$ matrix of the form $
\left(\begin{array}{c}
W_1\\ 
W_2
\end{array}\right) 
$. In the affine open neighborhood $\mathcal{U}_0$ of $\left(\begin{array}{c}
I_n\\ 
0
\end{array}\right) 
$ consisting of the $W$ such that $W_1$ is invertible we may write $W$ as $
\left(\begin{array}{c}
I_n\\ 
A
\end{array}\right) 
$, where $A = W_2W_1^{-1}$ is symmetric. Indeed, note that since $W$ is isotropic we have $W^{t}JW = 0$, that is $W_1^tW_2 = W_2^tW_1$ which in turn yields $(W_2W_1^{-1})^t = (W_1^{-1})^tW_1^tW_2W_1^{-1} = W_2W_1^{-1}$. Therefore, the Pl\"ucker embedding of $\G(n,V)$ restrict to an embedding
$$
\begin{array}{cccc}
\phi_{n,2n}: &\LG(n,2n)& \rightarrow & \P(V_{\omega_n})\subseteq\P(\bigwedge^{n}V)
\end{array}
$$
that is locally given by 
\stepcounter{thm}
\begin{equation}\label{embLGloc}
\begin{array}{cccc}
\phi_{n,2n|\mathcal{U}_0}: &\mathcal{U}_0\subset\LG(n,2n)& \longrightarrow & \P(V_{\omega_n})\subseteq\P(\bigwedge^{n}V)\\
      & W & \mapsto & (1,A,\bigwedge^2A,\dots, \bigwedge^nA)
\end{array}
\end{equation}
Here $\P(V_{\omega_n})$ is the linear span of $\LG(n,2n)$ in $\P(\bigwedge^{n}V)$, and $V_{\omega_n}$ can be identified with the irreducible representation of $\Sp(2n)$ with highest weight $\omega_n$, the fundamental weight associated to the last simple root $\alpha_n$. In particular, $\LG(n,2n)$ is a variety of dimension 
$$\dim(\LG(n,2n)) = \frac{n(n+1)}{2}$$
embedded in a projective space $\P(V_{\omega_n})$ of dimension
$$\dim(\P(V_{\omega_n}))= \frac{1}{2}\sum_{j=1}^n\left(\binom{n}{j}^2+\binom{n}{j}\right)$$
and $\G(n,V)\cap \P(V_{\omega_n}) = \LG(n,2n)$.

Now, let $V$ be a complex vector space of dimension $2n$, endowed with a non degenerate quadratic
form $Q$. As we said before, the variety parametrizing maximal isotropic, with respect to $Q$, subspaces of $V$ has two connected components $S_{+}$ and $S_{-}$ which are isomorphic. Their linear spans in the Pl\"ucker embedding are in direct sum, and we have a splitting $\bigwedge^nV = \bigwedge^nV_{+}\oplus \bigwedge^nV_{-}$ into spaces of the same dimension, and $S_{\pm} = \G(n,V)\cap\mathbb{P}(\bigwedge^nV_{\pm})$. 

Let $X_{Q}\subset\mathbb{P}^{2n-1}$ be the smooth quadric hypersurface associated to $Q$. Then there are two families of linear subspace of projective dimension $n-1$ contained in $X_Q$, and the varieties $S_{\pm}$ parametrize precisely these liner subspaces. The automorphism switching the two connected components of the orthogonal group $O(V)$ induces an isomorphism between $S_{+}$ and $S_{-}$. We will denote by $\mathcal{S}_n$ either of these two isomorphic varieties. Therefore, restricting the Pl\"ucker embedding we get an embedding
\stepcounter{thm}
\begin{equation}\label{embS1}
\begin{array}{cccc}
\beta_{n}: & \mathcal{S}_n & \rightarrow & \P(\bigwedge^nV_{+})\subseteq\P(\bigwedge^{n}V)
\end{array}
\end{equation}

However, the minimal embedding of $\mathcal{S}_n$ is an embedding in the projectivized half-spin representation \cite[Section 2.1]{Te05}, let us denote it by
\stepcounter{thm}
\begin{equation}\label{embS2}
\begin{array}{cccc}
\alpha_{n}: & \mathcal{S}_n & \rightarrow & \P(\Delta)
\end{array}
\end{equation}  
Now, we describe this minimal embedding in an affine chart. We fix a basis of $V$ such that the quadratic form in this basis is 
$$
Q = \left(\begin{array}{cc}
0 & I_{n}\\ 
I_{n} & 0
\end{array}\right) 
$$
and represent a point $U\in \mathcal{S}_n$ as a $2n\times n$ matrices of the form $
\left(\begin{array}{c}
U_1\\ 
U_2
\end{array}\right) 
$. In the affine open neighborhood $\mathcal{U}_0$ of $\left(\begin{array}{c}
I_n\\ 
0
\end{array}\right) 
$ consisting of the $U$ such that $U_1$ is invertible we may write $U$ as $
\left(\begin{array}{c}
I_n\\ 
B
\end{array}\right) 
$, where $B = U_2U_1^{-1}$ is skew-symmetric. Indeed, in this case $U^tQU = 0$ implies that $U_1^tU_2 = -U_2^tU_1$, which in turn yields $(U_2U_1^{-1})^t = -(U_1^{-1})^tU_1^tU_2U_1^{-1} = -U_2U_1^{-1}$. The minimal embedding $\alpha_n$ is given in this chart by 
\stepcounter{thm}
\begin{equation}\label{embS2loc}
\begin{array}{cccc}
\alpha_{n|\mathcal{U}_0}: & \mathcal{U}_0\subset\mathcal{S}_n & \longrightarrow & \P(\Delta)\\
& U & \mapsto & (1,\pf_{2j}(B))
\end{array}
\end{equation}
where $\pf_{2j}(B)$ denotes all the $2j\times 2j$ principal Pfaffians of $B$, and $j = 1,\dots,\frac{n}{2}$ if $n$ is even, while $j = 1,\dots,\frac{n-1}{2}$ if $n$ is odd. In particular $\mathcal{S}_n$ is a projective variety of dimension
$$\dim(\mathcal{S}_n) = \frac{n(n-1)}{2}$$ 
embedded in a projective space $\P(\bigwedge^nV_{+})$ of dimension 
$$\dim(\P(\bigwedge^nV_{+})) = \frac{1}{2}\binom{2n}{n}-1$$
via the embedding $\beta_n$ in (\ref{embS1}), and in a projective space $\P(\Delta)$ of dimension 
$$\dim(\P(\Delta)) = 2^{n-1}-1$$
via the minimal embedding $\alpha_n$ in (\ref{embS2}). Note that $\beta_n$ can be obtained composing $\alpha_n$ with a degree two Veronese embedding.
     
\section{Higher osculating spaces and projections}\label{osculating}
Let $X\subset \P^N$ be an integral projective variety of dimension $n$, $p\in X$ a smooth point, and 
$$
\begin{array}{cccc}
\phi: &\mathcal{U}\subseteq\mathbb{C}^n& \longrightarrow & \mathbb{C}^{N}\\
      & (t_1,\dots,t_n) & \mapsto & \phi(t_1,\dots,t_n)
\end{array}
$$
with $\phi(0)=p$, be a local parametrization of $X$ in a neighborhood of $p\in X$. 

For any $s\geq 0$ let $O^s_pX$ be the affine subspace of $\mathbb{C}^{N}$ passing through $p\in X$, and whose direction is given by the subspace generated by the vectors $\phi_I(0)$, where $I = (i_1,\dots,i_r)$ is a multi-index such that $|I|\leq s$ and 
$$
\phi_I = \frac{\partial^{|I|}\phi}{\partial t_1^{i_1}\dots\partial t_r^{i_r}}
$$

\begin{Definition}\label{oscdef}
The $s$-\textit{osculating space} $T_p^s X$ of $X$ at $p$ is the projective closure in $\mathbb{P}^N$ of the affine subspace $O^s_pX\subseteq \mathbb{C}^{N}$.
\end{Definition}

For instance, $T_p^0 X=\{p\}$, and $T_p^1 X$ is the usual tangent space of $X$ at $p$. When no confusion arises we will write $T_p^s$ instead of $T_p^sX$.

Note that while the dimension of the tangent space at a smooth point is always equal to the dimension of the variety, higher order osculating spaces can be strictly smaller than expected even at a general point. In general, we have
$$
\dim(T_p^s X) = \min\left\{\binom{n+s}{n}-1-\delta_{s,p},N\right\}
$$
where $\delta_{s,p}$ is the number of independent differential equations of order less or equal than $s$ satisfied by $X$ at $p$.

We will be interested in relating the osculating spaces of a projective variety to those of its linear sections. It is well-known that the tangent space of a transverse linear section of a variety at a smooth point is cut out in the tangent space of the variety by the linear space we used to intercept
the linear section. This is not always the case for higher order osculating spaces.

\begin{Definition}\label{Wellbehaved}
Let $X\subset\mathbb{P}^N$ be an irreducible variety and $Y=\mathbb{P}^k\cap X$ be a linear section of $X$. We say that $Y$ is \textit{osculating well-behaved} if for each smooth point $p\in Y$ we have 
$$T_p^sY=\mathbb{P}^k\cap T_p^sX$$ 
for every $s\geq 0$. 
\end{Definition}

In the following example we give a variety and a linear section of it that is not osculating well behaved.

\begin{Example}
In the projective space $\mathbb{P}^{k+2}$ consider two complementary subspaces $\mathbb{P}^1,\mathbb{P}^k$, and let $C\subset\mathbb{P}^k$ be a degree $k$ rational normal curve. Fixed an isomorphism $\psi:\mathbb{P}^1\rightarrow C$ we consider the rational normal scroll

$$S_{(1,k)} = \bigcup_{p\:\in\: \mathbb{P}^1}\left\langle p, \psi(p)\right\rangle\subset \mathbb{P}^{k+2}$$
where $\left\langle p, \psi(p)\right\rangle$ is the line through $p$ and $\psi(p)$. Then $S_{(1,k)}$ can be locally parametrized by the map
$$
\begin{array}{cccc}
\phi: & \mathbb{A}^1\times\mathbb{P}^1 & \longrightarrow & \mathbb{P}^{k+2}\\ 
 & (u,[\alpha_0:\alpha_1]) & \mapsto & [\alpha_0 u:\alpha_0:\alpha_1 u^k:\alpha_1 u^{k-1}:\dots:\alpha_1 u:\alpha_1].
\end{array} 
$$ 
Now, consider the Segre embedding 
$$
\begin{array}{cccc}
\sigma: & \mathbb{P}^1\times\mathbb{P}^k & \longrightarrow & \mathbb{P}^{2k+1}\\ 
 & ([u:v],[\alpha_0:\dots:\alpha_k]) & \mapsto & [\alpha_0u:\dots:\alpha_ku:\alpha_0v:\dots :\alpha_kv].
\end{array} 
$$ 
and let $\Sigma_{(1,k)}$ be its image. Note that $\Sigma_{(1,k)}$ is locally parametrized by
$$
\begin{array}{cccc}
\widetilde{\sigma}: & \mathbb{A}^1\times\mathbb{P}^k & \longrightarrow & \mathbb{P}^{2k+1}\\ 
 & ([u:1],[\alpha_0:\dots:\alpha_k]) & \mapsto & [\alpha_0u:\dots:\alpha_ku:\alpha_0:\dots :\alpha_k].
\end{array} 
$$ 
and that $\deg(\Sigma_{(1,k)}) = \deg(S_{(1,k)}) = k+1$. Now, take $\alpha_{i} = \alpha_1u^{i-1}$ for $i=2,\dots,k$. Then
$$\widetilde{\sigma}(u,[\alpha_0:\alpha_1:\dots :\alpha_1u^{k-1}]) = [\alpha_0u:\alpha_1u:\dots:\alpha_1u^{k-1}:\alpha_1u^k:\alpha_0:\alpha_1:\dots\alpha_1u^{k-1}]$$
and the coordinate functions of this last map are exactly the ones appearing in the expression of $\phi$. Therefore, if $[Z_0:\dots:Z_{2k+1}]$ are the homogeneous coordinates on $\mathbb{P}^{2k+1}$ and 
$$H^{k+2} = \{Z_j-Z_{k+j+2}=0,\: j = 1,\dots,k-1\}\cong\mathbb{P}^{k+2}$$
then we have 
$$S_{(1,k)} = \Sigma_{(1,k)}\cap H^{k+2}\subset\mathbb{P}^{2k+1}$$
By \cite[Example 4.12]{AMR17} we have that if $p\in S_{(1,k)}$ is a general point and $k\geq 2$ then $\dim(T_p^2S_{(1,k)}) = 4$. On the other hand, \cite[Corollary 2.6]{AMR17} yields $T_p^2\Sigma_{(1,k)} = \mathbb{P}^{2k+1}$. We conclude that
$$T_p^2S_{(1,k)}\subsetneqq T_p^2\Sigma_{(1,k)} \cap H^{k+2} = H^{k+2}$$
for all $k\geq 3$.
\end{Example}
    
\stepcounter{thm}
\subsection{Osculating spaces of Lagrangian Grassmannians}\label{sec_LG}
Let $\Lambda:=\{I\subset\{1,\ldots,2n\}\: ; \:|I|=n\}$, and given $I,J\in \Lambda$ define the distance between $I$ and $J$ as $d(I,J)=|I|-|I\cap J|$, which is also know as Hamming distance. We rewrite the parametrization (\ref{embLGloc}) as follows:
\stepcounter{thm}
\begin{equation}\label{par_LG}
\phi:\C^{\frac{n(n+1)}{2}}
\rightarrow\mathcal{LG}(n,2n)
\end{equation}
given by
$$
M=\left(
  \begin{array}{cccccc}
    1 & \cdots & 0 & a_{1,1} & \cdots & a_{1,n}\\
    \vdots & \ddots & \vdots &\vdots & \ddots& \vdots\\
    0 & \cdots & 1 &a_{1,n} & \cdots & a_{n,n}\\
      \end{array}
\right) \mapsto (\det(M_J))_{J\in\Lambda}
$$
where $M_J$ is the $n\times n$ matrix obtained from $M$ considering just the columns indexed by $J$.

Fix $J\in \Lambda$, and let $\sigma_J\in S_I$ be the permutation which changes the rows of $M_J$ so that the new matrix is of the form
\stepcounter{thm}
\begin{equation}\label{Minor}
\left(
    \begin{array}{cc}
      I_r & B \\
       0 & \overline{M}_J  \\
    \end{array}
  \right)
\end{equation}
where $r=|I_0\cap J|, I_0:=\{1,\dots,n\},$ and $\overline{M}_J $ is the submatrix  of $(a_{ij})_{i,j=1}^n$ given by the rows indexed by $I_0\setminus J$ and the columns indexed by $(J\setminus I_0)-n$. Note that $\det(M_J)=\epsilon(\sigma_J)\det\overline{M}_J$, where $\epsilon$ is the group homomorphism associating to a permutation its sign.

Now, note that if $J\in\Lambda$ and $J'\in \Lambda$ is given by
\stepcounter{thm}
\begin{equation}\label{J'}
J'=((I_0\setminus J)+n)\cup ((I_0^c\setminus J)-n)
\end{equation}
we have that $d(I_0,J)=d(I_0,J')=n-r$. Moreover $\det\overline{M}_J=\det\overline{M}_{J'}$, and thus $\epsilon(\sigma_J)\det(M_J)=\epsilon(\sigma_{J'})\det(M_{J'})$. We will denote by $\Sigma$ the subset of $\Lambda \times \Lambda$ given by
\stepcounter{thm}
\begin{equation}\label{bigsigma}
\Sigma:=\{(J,J')\in\Lambda \times \Lambda \: | \: J'=((I_0\setminus J)+n)\cup ((I_0^c\setminus J)-n)\}
\end{equation}
Furthermore, we define $\Sigma_s:=\{(J,J')\in \Sigma \: | \: d(I_0,J)\leq s\}$.

\begin{Proposition}\label{PropOscLG} 
For any $s\geq 0$ and $I\in \Lambda$ we have
$$
T_{e_I}^s\mathcal{LG}(n,2n) = \langle \epsilon(\sigma_J)e_J+\epsilon(\sigma_{J'})e_{J'} \mid (J,J')\in\Sigma_s\rangle = T_{e_I}^s\mathcal{G}(n,V)\cap \mathbb{P}(V_{\omega})
$$
In particular, $T_{e_I}^s\mathcal{LG}(n,2n) = \mathbb{P}(V_{\omega})$ for any $s\geq n$, and $\mathcal{LG}(n,2n) = \mathcal{G}(n,V)\cap \mathbb{P}(V_{\omega}) \subset \mathbb{P}(\bigwedge^n\mathbb{C}^{2n})$ is osculating well-behaved.
\end{Proposition}
\begin{proof}
We may assume that $I=\{1,\ldots,n\}\in \Lambda$ and use the parametrization \ref{par_LG}. First, note that each variable appears in degree at most two in the coordinates of $\phi$. Therefore, deriving three times with respect to the same variable always gives zero. Furthermore, since the degree of $\det(M_J)$ with respect to the $a_{i,j}$'s is at most $n$, all partial derivatives of order greater than or equal to $n+1$ are zero. Then it is sufficient to take $s\leq n$. Note also that $\det(M_J)$ is a homogeneous polynomial of degree $m=d(I,J)$ in the variables $a_{i,j}$, its partial derivatives of order $s$ are zero if $s>m$ and a homogeneous polynomial of degree $m-s$ otherwise. Therefore, the partial derivatives of $\det(M_J)$ evaluated at zero vanish for $s\neq m$.

Let us fix $J=\{j_1,\ldots,j_n\}\in \Lambda$ and take $k,k'\in\{1,\ldots,n\}$. We will compute $\frac{\partial \det(M_J)}{\partial a_{k,k'}}$. Remember that the variable $a_{k,k'}$ appears at most two times in the entries of $M_J$.

In order to compute $\frac{\partial \det(M_J)}{\partial a_{k,k'}}$ when $a_{k,k'}$ appears in the entries of $M_J$ we will use the Laplace expansion of $\det(M_J)$ with respect the $k$-th row. If we denote the entries of $M$ by $m_{i,j}$, then the entries of $M_J$ are given by $m_{i,j_l}$ with $l=1,\dots,n$. Therefore, the variable $a_{k,k'}$ appears in the entries of $M_J$ if and only if $a_{k,k'}=m_{k,k'+n}$, in this case we set $k'+n=j_l\in J$. If $a_{k,k'}$ appears at least once in the entries of $M_J$, then the Laplace expansion of $\det(M_J)$ is given by
\stepcounter{thm}
\begin{equation}\label{Lapl_Exp}
\begin{array}{ccl}
\det(M_J)&=&\sum_{i=1}^n(-1)^{k+i}m_{k,j_i}\det(M_{J,\widehat{k},\widehat{j}_i})\\
&=&(-1)^{k+l}m_{k,j_l}\det(M_{J,\widehat{k},\widehat{j}_l})+\sum_{i=1;i\neq l}^n(-1)^{k+i}m_{k,j_i}\det(M_{J,\widehat{k},\widehat{j}_i})\\
&=&(-1)^{k+l}a_{k,k'}\det(M_{J,\widehat{k},\widehat{k'+n}})+\sum_{i=1;i\neq l}^n(-1)^{k+i}m_{k,j_i}\det(M_{J,\widehat{k},\widehat{j}_i})
\end{array}
\end{equation}
where $M_{J,\widehat{k},\widehat{j}_i}$ denotes the submatrix of $M_J$ obtained deleting the row indexed by $k$ and the column indexed by $j_i.$

Now, we derive the equation (\ref{Lapl_Exp}) with respect to $a_{k,k'}$. If the variable $a_{k,k'}$ appears just once in the entries of $M_J$, then the minors $M_{J,\widehat{k},\widehat{j}_i}$, with $i\neq l$, do not depend on $a_{k,k'}$. Hence, in this case we have
$$
\frac{\partial \det(M_J)}{\partial a_{k,k'}} = (-1)^{k+l}\det(M_{J,\widehat{k},\widehat{k'+n}})\frac{\partial a_{k,k'}}{\partial a_{k,k'}} = (-1)^{k+l}\det(M_{J,\widehat{k},\widehat{k'+n}})
$$
Note that $\det(M_{J,\widehat{k},\widehat{k'+n}})$ vanishes when the determinant of $M_J$ does not depend on $a_{k,k'}$. Now, suppose that $a_{k,k'}$ appears twice in the entries of $M_J$, in this case $a_{k,k'}$ appears in all minors $M_{J,\widehat{k},\widehat{j}_l}$, with $i\neq l$, except $M_{J,\widehat{k},\widehat{k+n}}$. Then using the notation $k+n=j_r\in J$ from (\ref{Lapl_Exp}) we get that $\det(M_J) = (-1)^{k+l}a_{k,k'}\det(M_{J,\widehat{k},\widehat{k'+n}})+(-1)^{k+r}a_{k,k}\det(M_{J,\widehat{k},\widehat{k+n}})+\sum_{i=1;i\neq l,r}^n(-1)^{k+i}m_{k,j_i}\det(M_{J,\widehat{k},\widehat{j}_i})$.

Now, note that using the Laplace expansion of $\det(M_{J,\widehat{k},\widehat{j}_i})$ with respect the $k'$-th row, the partial derivative of $\det(M_{J,\widehat{k},\widehat{j}_i})$, with $i\neq l,r,$ with respect to $ a_{k,k'}$ can be written as $\frac{\partial\det(M_{J,\widehat{k},\widehat{j}_i})}{\partial a_{k,k'}}=(-1)^{k'+r}\det(M_{J,\widehat{k},\widehat{k'},\widehat{j}_i,\widehat{k+n}})$, where $M_{J,\widehat{k},\widehat{k'},\widehat{j}_i,\widehat{k+n}}$ denotes the submatrix of $M_J$ obtained deleting the rows $k$ and $k'$ and the columns $j_i$ and $k+n$. Therefore,
$$
\begin{array}{ccl}
\frac{\partial \det(M_J)}{\partial a_{k,k'}}&=&(-1)^{k+l}\frac{\partial}{\partial a_{k,k'}}a_{k,k'}\det(M_{J,\widehat{k},\widehat{k'+n}})+(-1)^{k+r}\frac{\partial}{\partial a_{k,k'}}a_{k,k}\det(M_{J,\widehat{k},\widehat{k+n}})\\
 && +\sum_{i=1;i\neq l,r}^n(-1)^{k+i}m_{k,j_i}\frac{\partial \det(M_{J,\widehat{k},\widehat{j}_i})}{\partial a_{k,k'}}\\
 &=&(-1)^{k+l}\det(M_{J,\widehat{k},\widehat{k'+n}})+(-1)^{k+l}a_{k,k'}\frac{\partial\det(M_{J,\widehat{k},\widehat{k'+n}})}{\partial a_{k,k'}}+0\\
  &&+\sum_{i=1;i\neq l,r}^n(-1)^{k+i}m_{k,j_i}\frac{\partial\det(M_{J,\widehat{k},\widehat{j}_i})}{\partial a_{k,k'}}\\
  &=&(-1)^{k+l}\det(M_{J,\widehat{k},\widehat{k'+n}})+\sum_{i=1;i\neq r}^n(-1)^{k+i}m_{k,j_i}\frac{\partial \det(M_{J,\widehat{k},\widehat{j}_i})}{\partial a_{k,k'}}\\
  &=&(-1)^{k+l}\det(M_{J,\widehat{k},\widehat{k'+n}})+\sum_{i=1;i\neq r}^n(-1)^{k+i}m_{k,j_i}(-1)^{k'+r}\det(M_{J,\widehat{k},\widehat{k'},\widehat{j}_i,\widehat{k+n}})\\
   &=&(-1)^{k+l}\det(M_{J,\widehat{k},\widehat{k'+n}})+(-1)^{k'+r}\sum_{i=1;i\neq r}^n(-1)^{k+i}m_{k,j_i}\det(M_{J,\widehat{k},\widehat{k'},\widehat{j}_i,\widehat{k+n}})\\
   &=&(-1)^{k+l}\det(M_{J,\widehat{k},\widehat{k'+n}})+(-1)^{k'+r}\det(M_{J,\widehat{k'},\widehat{k+n}})
\end{array}
$$ 
Again, the minors $\det(M_{J,\widehat{k},\widehat{k'+n}})$ and $\det(M_{J,\widehat{k'},\widehat{k+n}})$ may vanish.

Now, consider the derivatives of order $m>1$, set $\alpha=\{\alpha_1,\ldots,\alpha_m\}\subset I$, $\beta=\{\beta_1,\ldots,\beta_m\}\subset I$, and define 
$$M_{J,\widehat{\alpha},\widehat{\beta}}=\left\{\begin{array}{cl}
M_{J,\widehat{\alpha}_1,\ldots,\widehat{\alpha}_m,\widehat{j}_{\gamma_1},\ldots,\widehat{j}_{\gamma_m}} & \text{ if  } \alpha\subset I\setminus J,\beta_i+n=j_{\gamma_i}\in J, \: |\alpha| = |\beta| = m\\
0 & \text{ otherwise}
\end{array}\right.$$

Now, take $K = \{k_1,\ldots,k_m\}\subset I$, $K'= \{k_1',\ldots,k_m'\}\subset I$, and for each $\alpha=\{\alpha_1,\ldots,\alpha_m\}\subset K\cup K'$ with $\alpha_i\in\{k_i,k_i'\}$ define $\alpha^*=\{\alpha_1^*,\ldots,\alpha_m^*\}$, where $\{\alpha_i,\alpha_i^*\}=\{k_i,k_i'\}$, and
$$\Delta_J=\{\alpha\: |\: \det(M_{J,\widehat{\alpha},\widehat{\alpha}^*})\neq 0\}$$ 
Therefore
$$\frac{\partial^m\det(M_J)}{\partial a_{k_1,k_1'}\dots\partial a_{k_m,k_m'}}=                                
\sum_{\alpha\in\Delta_J}(-1)^{\sum\alpha_i+ \gamma_i}\det(M_{J,\widehat{\alpha},\widehat{\alpha}^*})$$
where $\gamma_i$ is such that $\alpha_i^*+n=j_{\gamma_i}\in J$. Remember that this partial derivative evaluated at zero vanish when $d(I,J)\neq m$. In this case the matrix $M_{J,\widehat{\alpha},\widehat{\alpha}^*}$ is obtained deleting the rows indexed by $I\setminus J$ and the columns indexed by $J\setminus I$, hence $M_{J,\widehat{\alpha},\widehat{\alpha}^*}=I_{n-m}$. Therefore, for all $\alpha\in\Delta_J$ we have
$$\sum_{i=1}^m\gamma_i=(n-m+1)+\cdots+(n-1)+n \mbox{ and } \sum_{i=1}^m\alpha_i=\sum_{i\in I\setminus J}i$$
and we conclude that $\sum_{i=1}^m\alpha_i+\gamma_i$ does not depend on $\alpha\in \Delta_J$. Now, setting $c_J=|\Delta_J|$ we have
$$\frac{\partial^m\det(M_J)}{\partial a_{k_1,k_1'}\dots\partial a_{k_m,k_m'}}(0)=
\left\{\begin{array}{cl}
\pm c_J&\text{if  } K'+n\cup (I\setminus K)=J\\
  0 & \text{otherwise}
\end{array}\right.
$$
Note that $K'+n\cup (I\setminus K)=J$ if and only if $K+n\cup (I\setminus K')=J'$, and in this case $(J,J')\in \Sigma_m$. Hence, we conclude that
$$
\frac{\partial^m\phi}{\partial a_{k_1,k_1'}\dots\partial a_{k_m,k_m'}}(0)=
\left\{\begin{array}{cl}
\epsilon(\sigma_J)e_J+\epsilon(\sigma_{J'})e_{J'} &\text{if  } K'+n\cup (I\setminus K)=J\\
  0 & \text{otherwise}
\end{array}\right.
$$
and thus 
$$\left\langle\frac{\partial^m\phi}{\partial a_{K,K'}}(0)\mid K,K'\subseteq I, |K| = |K'|=m\right\rangle = \langle\epsilon(\sigma_J)e_J+\epsilon(\sigma_{J'})e_{J'}\mid (J,J')\in\Sigma_m\setminus\Sigma_{m-1}\rangle$$
Therefore
\stepcounter{thm}
\begin{equation}\label{osc}
T_{e_I}^s\mathcal{LG}(n,2n) = \left\langle\frac{\partial^m\phi}{\partial a_{K,K'}}(0) \mid 0\leq m\leq s\right\rangle = \langle \epsilon(\sigma_J)e_J+\epsilon(\sigma_{J'})e_{J'}\mid (J,J')\in\Sigma_s\rangle 
\end{equation}
Now, set
\stepcounter{thm}
\begin{equation}\label{sym}
\mathbb{P}^{N-M}:=\{\epsilon(\sigma_J)P_J-\epsilon(\sigma_{J'})P_{J'}=0\mid(J,J')\in \Sigma \text{ and } J\neq J'\}\subseteq\mathbb{P}^N
\end{equation}
where $M=\frac{1}{2}\sum_{k=1}^n\binom{n}{k}^2-\binom{n}{k}$. Now, (\ref{osc}) and (\ref{sym}) yield that 
$$
T_{e_I}^s\mathcal{LG}(n,2n) = \{P_J=0\mid J\in \Lambda\text{ and } d(I,J)>s\}\cap\mathbb{P}^{N-M} = T_{e_I}^s\mathcal{G}(n,V)\cap\mathbb{P}^{N-M}
$$
Finally, we get that
$$T_{e_I}^s\mathcal{LG}(n,2n)\subset T_{e_I}^s\mathcal{G}(n,V)\cap \mathbb{P}(V_{\omega})\subset T_{e_I}^S\mathcal{G}(n,V)\cap\mathbb{P}^{N-M}=T_{e_I}^s\mathcal{LG}(n,2n)$$
and hence $T_{e_I}^s\mathcal{LG}(n,2n) = T_{e_I}^s\mathcal{G}(n,V)\cap \mathbb{P}(V_{\omega})$.
\end{proof}

\begin{Corollary}\label{dimOscLG}
For any $p\in \mathcal{LG}(n,2n)\subseteq\mathbb{P}(V_{\omega_n})$ we have 
$$\dim(T_{p}^s\mathcal{LG}(n,2n)) = \frac{1}{2}\sum_{k=1}^{s}\binom{n}{k}\left(\binom{n}{k}+1\right)$$
for $1\leq s\leq n-1$ while $T_{p}^s\mathcal{LG}(n,2n) = \mathbb{P}(V_{\omega})$ for $s\geq n$.
\end{Corollary}
\begin{proof}
The symplectic group $Sp(2n)$ acts transitively on $\mathcal{LG}(n,2n)$ and hence $\dim(T_{p}^s\mathcal{LG}(n,2n)) = T_{e_I}^s\mathcal{LG}(n,2n)$ for any $p\in \mathcal{LG}(n,2n)$. Now, the claim follows from Proposition \ref{PropOscLG}.
\end{proof}

\stepcounter{thm}
\subsection{Osculating spaces of Spinor varieties in the Pl\"ucker embedding}\label{oscLGSec}
Let us write the parametrization induced	 by the embedding $\beta_n$ in (\ref{embS1}) as follows:
\stepcounter{thm}
\begin{equation}\label{parSplu}
\beta:\C^{\frac{n(n-1)}{2}}\rightarrow \mathcal{S}_n\subset \mathbb{P}(\bigwedge^nV_{+})
\end{equation}
given by
$$
M=\left(
  \begin{array}{cccccc}
    1 & \cdots & 0 & 0 & \cdots & a_{1,n}\\
    \vdots & \ddots & \vdots &\vdots & \ddots& \vdots\\
    0 & \cdots & 1 &-a_{1,n} & \cdots & 0\\
      \end{array}
\right)\mapsto (\det(M_J))_{J\in\Lambda}
$$
where $M_J$ is the $n\times n$ matrix obtained from $M$ considering just the columns indexed by $J$.

Fix $J\in \Lambda$, and let $\sigma_J$ and $\overline{M}_J$ as in (\ref{Minor}). Then $\det(M_J)=\epsilon(\sigma_J)\det\overline{M}_J$.

Now, set $J'\in \Lambda$ as in (\ref{J'}), then $\det\overline{M}_J = -\det\overline{M}_{J'}$ for each $J,J'\in \Lambda$ with $J\neq J'$ that satisfies (\ref{J'}), and $\det\overline{M}_J=0$ if $d(I,J)$ is odd and $J'=J$ in (\ref{J'}). Thus we have
$$\epsilon(\sigma_J)\det(M_J)=-\epsilon(\sigma_{J'})\det(M_{J'})$$
for each $J,J'\in \Lambda$ with $J\neq J'$ satisfying (\ref{J'}). 

Consider $\Sigma$ as in (\ref{bigsigma}).
For each $(J,J')\in \Sigma$ with $J\neq J'$ consider the hyperplane $V(\epsilon(\sigma_J)P_J+\epsilon(\sigma_{J'})P_{J'})$, and for each $(J,J)\in \Sigma$ with $d(I,J)$ odd, consider the hyperplane $ V(P_J)$. We have that
$$\mathbb{P}(\bigwedge^nV_{+}) =\bigcap_{{\tiny\begin{array}{c}(J,J')\in \Sigma \text{ and }J\neq J'\\
\text{or}\\
(J,J)\in \Sigma\text{ and }d(I,J)\text{ is odd}
\end{array}}}V(\epsilon(\sigma_J)P_J+\epsilon(\sigma_{J'})P_{J'})$$
and hence $\mathcal{S}_n=\mathcal{G}(n,V)\cap\mathbb{P}(\bigwedge^nV_{+})$. Now, define  
$$\Sigma_s:=\{(J,J')\in \Sigma \: | \: d(I,J)\leq s\text{ and } J\neq J'\}\cup\{(J,J)\in \Sigma \: | \: d(I,J)\leq s \text{ is even}\}$$

\begin{Proposition}\label{Osc_S_Pl} 
For any $s\geq0$ and $I\in \Lambda$ we have
$$T_{e_I}^s\mathcal{S}_n = \langle\{e_J \: | \: (J,J)\in\Sigma_s\}\cup\{\epsilon(\sigma_J)e_J-\epsilon(\sigma_{J'})e_{J'}\: | \: (J,J')\in\Sigma_s\}\rangle$$
In particular, $T_{e_I}^s\mathcal{S}_n = T_{e_I}^s\mathcal{G}(n,V)\cap\mathbb{P}(\bigwedge^nV_{+})$, $T_{e_I}^s\mathcal{S}_n = \mathbb{P}(\bigwedge^nV_{+})$ for $s\geq 2\lfloor\frac{n}{2}\rfloor$, and $\mathcal{S}_n=\mathcal{G}(n,V)\cap\mathbb{P}(\bigwedge^nV_{+})$ is osculating well-behaved.
\end{Proposition}
\begin{proof}
The proof is analogous to the one of Proposition \ref{PropOscLG}. It is enough to use the parametrization (\ref{parSplu}) and to note that in this case
$$\frac{\partial^m\beta}{\partial a_{k_1,k_1'} \cdots\partial a_{k_m,k_m'}}(0)=
\left\{\begin{array}{cl}
\epsilon(\sigma_J)e_J-\epsilon(\sigma_{J'})e_{J'}& \text{if } (J,J')\in\Sigma_m\setminus \Sigma_{m-1}\text{ and }J\neq J'\\
e_J& \text{if } (J,J)\in\Sigma_m\setminus \Sigma_{m-1}\\
0&\text{otherwise}
\end{array}\right.
$$
and thus
$$\left\langle\frac{\partial^m\beta}{\partial a_{K,K'}}(0) \right\rangle= \langle\{e_J\text{ : }(J,J)\in\Sigma_m\setminus\Sigma_{m-1}\}\cup\{\epsilon(\sigma_J)e_J-\epsilon(\sigma_{J'})e_{J'}\text{ : }(J,J')\in\Sigma_m\setminus\Sigma_{m-1}\}\rangle 
$$
Therefore
$$
T_{e_I}^s\mathcal{S}_n = \left\langle\frac{\partial^m\beta}{\partial a_{K,K'}}(0)\mid 0\leq m\leq s\right\rangle = \langle\{e_J\: | \: (J,J)\in\Sigma_s\}\cup\{\epsilon(\sigma_J)e_J-\epsilon(\sigma_{J'})e_{J'}\: | \: (J,J')\in\Sigma_s\}\rangle 
$$
and hence
$$
T_{e_I}^s\mathcal{S}_n = \{P_J=0\mid J\in \Lambda\text{ and } d(I,J)>s\}\cap\mathbb{P}^{N-M} = T_{e_I}^s\mathcal{G}(n,V)\cap\mathbb{P}^{N-M}
$$
where $M = \frac{1}{2}\sum_{i=1}^n\binom{n}{k}^2-\binom{n}{k}+\sum_{k=1}^{\lceil\frac{n}{2}\rceil}\binom{n}{2k-1}$.
\end{proof}

\begin{Corollary}\label{dim_Osc_S1}
For any $p\in \mathcal{S}_n\subseteq\mathbb{P}(\bigwedge^nV_{+})$ we have 
$$\dim(T_{p}^s\mathcal{S}_n) = \frac{1}{2}\sum_{k=1}^{s}\binom{n}{k}\left(\binom{n}{k}+1\right)-\sum_{k=1}^{\lceil\frac{s}{2}\rceil}\binom{n}{2k-1}$$
for $1\leq s\leq 2\lfloor\frac{n}{2}\rfloor-1$ while $T_{p}^s\mathcal{S}_n = \mathbb{P}(\bigwedge^nV_{+})$ for $s\geq 2\lfloor\frac{n}{2}\rfloor$.
\end{Corollary}
\begin{proof}
The special orthogonal group $SO(2n)$ acts transitively on $\mathcal{S}_n$. Now, the claim follows arguing as in Corollary \ref{dimOscLG}, using Proposition \ref{Osc_S_Pl} instead of Proposition \ref{PropOscLG}.
\end{proof}

\begin{say}\label{repth}
Let $G$ be a connected semisimple complex algebraic group, $\mathfrak{g}$ its Lie algebra, and 
$$U\mathfrak{g} = \frac{T\mathfrak{g}}{(x\otimes y - y\otimes x - [x,y], x,y\in \mathfrak{g})}$$
the universal enveloping algebra. Note that $U\mathfrak{g}$ has a natural filtration $U\mathfrak{g}^0 \subset U\mathfrak{g}^l\subset\dots$ such that $U\mathfrak{g}^s$ is spanned by products $g_1 \dots g_l$, where $g_i\in\mathfrak{g}$ and $l\leq s$.

Let $V_{\lambda}$ be an irreducible $G$-module with the highest weight $\lambda$, and $v_{\lambda}\in V_{\lambda}$ a highest weight vector. The action of $U\mathfrak{g}$ on $V_{\lambda}$ induces a filtration $V_\lambda^0\subset V_\lambda^1\subset \dots$ 
of $V_{\lambda}$
such that $V_{\lambda}^s = U\mathfrak{g}  ^sv_{\lambda}$.

Let $x_{\lambda}\in \mathbb{P}(V_{\lambda})$ be the point corresponding to $v_{\lambda}\in V_{\lambda}$, and $X = G/P \subseteq \mathbb{P}(V_{\lambda})$ the orbit of $x_{\lambda}$ via $G$, where $P$ is a parabolic subgroup, namely the stabilizer of $x_{\lambda}$ in $G$. If $\mathfrak{p}$ is the Lie algebra of $P$ then we can identify $T^1_{x_{\lambda}}X$ with $\mathfrak{g}/\mathfrak{p}$. Furthermore, by \cite[Proposition 2.3]{LM03} we have that $T^s_{x_{\lambda}}X = \mathbb{P}(V_{\lambda}^s)$ for any $s\geq 1$.
\end{say}

\begin{Corollary}\label{cor_rep_th}
Let $\mathcal{G}(n,V), \mathcal{LG}(n,2n), \mathcal{S}_n\subseteq\mathbb{P}(\bigwedge^nV)$ be the orbits of $p = [e_1\wedge\dots\wedge e_n]$ respectively under the action of $GL(V),Sp(V),SO(V)$, and let $\mathbb{P}(V_{\omega_n}),\mathbb{P}(\bigwedge^nV_{+})\subseteq\mathbb{P}(\bigwedge^nV)$ the irreducible representations generated respectively by $Sp(V),SO(V)$.

We denote by $V_{\mathcal{G}}^s,V_{\mathcal{LG}}^s,V_{\mathcal{S}}^s$ the pieces of the filtration described in (\ref{repth}) for $\bigwedge^nV, V_{\omega_n}, \bigwedge^nV_{+}$ respectively. Then we have
$$\mathbb{P}(V_{\mathcal{LG}}^s) = \mathbb{P}(V_{\mathcal{G}}^s)\cap \mathbb{P}(V_{\omega_n}),\quad \mathbb{P}(V_{\mathcal{S}}^s) = \mathbb{P}(V_{\mathcal{G}}^s)\cap \mathbb{P}(\bigwedge^nV_{+})$$
for any $s\geq 1$.
\end{Corollary}
\begin{proof}
It is a consequence of \cite[Proposition 2.3]{LM03} and Propositions \ref{PropOscLG}, \ref{Osc_S_Pl}.
\end{proof}

\begin{Remark}
As observed in Section \ref{IsG}, if $d = 2m+1$ and $r < m$ or $d = 2m$ and $r < m-1$ then $\bigwedge^rV$ is an irreducible $SO(V)$-module. Therefore, in these cases, a result analogous to Corollary \ref{cor_rep_th} can not hold for the embedding  $\mathcal{G}_{Q}(r,V)\subseteq\mathcal{G}(r,V)\subseteq\bigwedge^rV$. 
\end{Remark}

\stepcounter{thm}
\subsection{Osculating spaces of Spinor varieties in the Spinor embedding}
Let us denote by $\Gamma$ the following set
\stepcounter{thm}
\begin{equation}\label{Gamma}
\Gamma=\{I=\{i_1,\ldots,i_r\}\subset \{1,\ldots,2n\} \: | \: r \text{ is even and }0\leq r\leq 2n\}
\end{equation} 
when $r=0$ we set, by convention, $I=\varnothing$. Now, let us consider the parametrization
\stepcounter{thm}
\begin{equation}\label{parSS}
\alpha:\mathbb{C}^{\frac{n(n-1)}{2}}\rightarrow \mathcal{S}_n\subseteq \mathbb{P}^{2^{n-1}-1}
\end{equation} 
defined by
$$(I_n, A)=\left(
  \begin{array}{ccccccc}
    1      & \cdots & 0      & 0       &a_{1,2}    & \cdots  & a_{1,n}\\
    0      & \cdots & 0      & -a_{1,2} & 0        &\ddots   & \vdots\\
    \vdots & \ddots & \vdots &\vdots   & \ddots   & \ddots  & a_{n-1,n}\\
    0      & \cdots & 1      &-a_{n,1}  & \cdots   & -a_{n-1,n}       &0       \\
      \end{array}
\right)\mapsto (\pf(A_I))_{I\in\Gamma}$$
induced by the embedding $\alpha_n$ in (\ref{embS2}), where $A_I$ is the submatrix of $A$ obtained considering the rows and the columns indexed by $I\in \Gamma$.

For $I\in \Gamma$ we will denote by $P_I$ the corresponding homogeneous coordinate of $\mathbb{P}^{2^{n-1}-1}$, and by $e_I\in \mathbb{P}^{2^{n-1}-1}$ the coordinate point given by $P_J=0$ for $I\neq J$.

\begin{Proposition}\label{osc_S2}
For any $s\geq 0$ and $I\in \Gamma$ we have
$$T_{e_I}^s\mathcal{S}_n = \bigl\langle e_J \: | \: |J|\leq 2s\bigl\rangle = \bigl\{P_J = 0 \: |\: |J|> 2s\bigl\}$$
In particular, $T_{e_I}^s\mathcal{S}_n = \mathbb{P}(\Delta)$ for $s\geq\left\lfloor\frac{n}{2}\right\rfloor$.
\end{Proposition}
\begin{proof}
We can suppose that $I=\varnothing\in \Gamma$ and use the parametrization (\ref{parSS}). First, note that each variable appears in degree at most one the expression of $\alpha$. Thus, deriving two times with respect to the same variable always gives zero. Furthermore, since the degree of $\pf(A_I)$ with respect to the $a_{i,j}$'s is at most $\left\lfloor\frac{n}{2}\right\rfloor$ all partial derivatives of order grater than or equal to $\left\lfloor\frac{n}{2}\right\rfloor+1$ vanish.

Let $I=\{i_1,\ldots,i_r\}\in \Gamma$ with $r>0$ and $k,k'\in \{1,\ldots,n\}$ with $k\neq k'$. Then
$$\frac{\partial \pf(A_I)}{\partial a_{k,k'}}=\left\{\begin{array}{cl}
(-1)^{k+k'+1} \pf(A_{\widehat{k},\widehat{k'}}) & \text{if  } k,k'\in I\\
0 & \text{otherwise}
\end{array}\right.$$

In general, let $m\geq 1$ and $K=\{k_1,\ldots,k_m\}\subset\{1,\ldots,n\}, K'=\{k_1',\ldots,k_m'\}\subset\{1,\ldots,n\}$ such that $k_j\neq k_j'$ for all $1\leq j\leq m$, $K\cap K'=\emptyset$ and $|K|=|K'|=m$. Then
$$\frac{\partial^m \pf(A_I)}{\partial a_{k_1,k_1'}\cdots \partial a_{k_m,k_m'}}=\left\{\begin{array}{cl}
\pm \pf(A_{\widehat{K},\widehat{K'}})& \text{if  }K,K'\subset I\\
0 & \text{otherwise}
\end{array}\right.$$

Therefore,
$$\frac{\partial^m \pf(A_I)}{\partial a_{k_1,k_1'}\cdots \partial a_{k_m,k_m'}}(0)=\left\{\begin{array}{cl}
\pm 1 & \text{if  } K\dot{\cup}K' = I\\
0 & \text{otherwise}
\end{array}\right.$$
and
$$\frac{\partial^m\alpha}{\partial a_{k_1,k_1'}\cdots \partial a_{k_m,k_m'}}(0)=\pm e_{K\dot{\cup}K'}$$
Finally, we get $T_{e_I}^s\mathcal{S}_n=\left\langle\frac{\partial^m\alpha}{\partial a_{K,K'}}(0)\: | \: |K|=|K'|=m\leq s\rangle=\langle \pm e_I \: | \: |I|\leq 2s\right\rangle$.
\end{proof}

\begin{Corollary}\label{dim_osc_S2}
For any point $p\in\mathcal{S}_n\subseteq\mathbb{P}(\Delta)$ we have
$$\dim T_p^s\mathcal{S}_n = \sum_{k=1}^s\binom{n}{2k}$$
for $0\leq s\leq\left\lfloor\frac{n}{2}\right\rfloor-1$ while $T_p^s\mathcal{S}_n=\mathbb{P}^{2^{n-1}-1} $ for $s\geq \left\lfloor\frac{n}{2}\right\rfloor$.
\end{Corollary}
\begin{proof}
It is enough to argue as in the proof of Corollary \ref{dim_Osc_S1}, using Proposition \ref{osc_S2} instead of Proposition \ref{Osc_S_Pl}.
\end{proof}

\stepcounter{thm}
\subsection{Projections}
We will denote by $\Pi_{T_p^s}:X\dasharrow \mathbb{P}^{N_s}$ the linear projection of an irreducible projective variety $X\subset\mathbb{P}^N$ with center $T_p^sX$. Our aim is to establish, for Lagrangian Grassmannians and Spinor varieties, when such a projection is birational.

\begin{Proposition}\label{proj_osc}
Consider the Lagrangian Grassmannian $\mathcal{LG}(n,2n)\subset\mathbb{P}(V_{\omega_n})$. If $0\leq s\leq n-2$ then $\Pi_{T_p^s}$ is birational.

Now, consider the Spinor variety $\mathcal{S}_n\subset\mathbb{P}(\bigwedge^nV_{+})$ in the Pl\"ucker embedding. If $0\leq s \leq 2\lfloor \frac{n}{2}\rfloor-2$ then $\Pi_{T_p^s}$ is birational.

Finally, consider the Spinor variety $\mathcal{S}_n\subset\mathbb{P}(\Delta)$ in the Spinor embedding. If $0\leq s\leq \lfloor\frac{n}{2}\rfloor-2$ then $\Pi_{T_p^s}$ is birational.
\end{Proposition}
\begin{proof}
Consider the case of the Lagrangian Grassmannian. It is enough to prove that $\Pi_{T_{e_I}^{n-2}}$ is birational. By the description of the local chart for $\mathcal{LG}(n,2n)$ in (\ref{embLGloc}) and Proposition \ref{PropOscLG} we see that $\Pi_{T_{e_I}^{n-2}}$ keeps track of all the $(n-1)\times (n-1)$ minors of the symmetric matrix $A$, and of its determinant as well. Note that if $A$ is general then with these data we can reconstruct the inverse $A^{-1}$, and therefore $A$ itself. 

Now, consider $\mathcal{S}_n\subset\mathbb{P}(\bigwedge^nV_{+})$, and the parametrization (\ref{parSplu}). With the same notation as in Section \ref{oscLGSec} set $M = (I,B)$, where $B$ is an $n\times n$ skew-symmetric matrix. If $n$ is even then $2\lfloor \frac{n}{2}\rfloor-2 = n-2$, for $B$ general $\det(B)\neq 0$, and we can argue as in the previous case, applying Proposition \ref{Osc_S_Pl}, to prove that $\Pi_{T_p^{n-2}}$ is birational. On the other hand, if $n$ is odd then $\det(B)=0$. In this case $2\lfloor \frac{n}{2}\rfloor-2 = n-3$, and by Proposition \ref{Osc_S_Pl} $\Pi_{T_{e_I}^{n-3}}$ keeps track of all the $(n-1)\times (n-1)$ and $(n-2)\times (n-2)$ minors of $B$. Now, $n-1$ is even, and arguing as in the even case we can reconstruct all the $(n-1)\times (n-1)$ submatrices of $B$, and hence $B$ itself.    

Finally, consider $\mathcal{S}_n\subset\mathbb{P}(\Delta)$, and the parametrization (\ref{parSS}). Recall that if $A$ is an invertible $n\times n$ skew-symmetric matrix then $A^{-1}$ is given by

$$A^{-1}= \frac{1}{\pf(A)}\left(\begin{array}{ccccc}
0& -\pf(A_{1,2}) & \pf(A_{1,3}) & \cdots & -\pf(A_{1,n})\\
\pf(A_{1,2}) & 0  & -\pf(A_{2,3}) & \hdots & \pf(A_{2,n})\\
-\pf(A_{1,3}) & \pf(A_{2,3}) & 0 & \ddots & \vdots\\
\vdots & \vdots & \ddots & \ddots & -\pf(A_{n-1,n})\\
\pf(A_{1,2}) &  -\pf(A_{2,n}) & \cdots & \pf(A_{n-1,n}) & 0\\
\end{array}\right)$$
where $A_{i,j}$ denotes the submatrix of $A$ obtained removing from $A$ the rows and columns indexed by $i$ and $j$ \cite[Section 3]{Kr16}. 

If, $n$ is even then by (\ref{parSS}) and Proposition \ref{osc_S2} we see that $\Pi_{T_{e_I}^{\lfloor\frac{n}{2}\rfloor-2}}$ keeps track of $\pf(A)$ and of all the $(n-2)\times (n-2)$ sub-Pfaffians of $A$. Similarly, if $n$ is odd (\ref{parSS}) and Proposition \ref{osc_S2} yield that $\Pi_{T_{e_I}^{\lfloor\frac{n}{2}\rfloor-2}}$ keeps track of all the $(n-1)\times (n-1)$ and $(n-3)\times (n-3)$ sub-Pfaffians of $A$. Therefore, in both cases to conclude that $\Pi_{T_{e_I}^{\lfloor\frac{n}{2}\rfloor-2}}$ is birational it is enough to argue as we did for $\mathcal{S}_n\subset\mathbb{P}(\bigwedge^nV_{+})$ using the expression for $A^{-1}$ above.   
\end{proof}

\begin{Remark}\label{rem1}
Proposition \ref{proj_osc} can not be improved. A standard computation shows that the projection $\Pi_{T_{e_I}}:\mathcal{S}_{3}\subset\mathbb{P}^9\dasharrow\mathbb{P}^5$ contracts $\mathcal{S}_{3}$ onto the Veronese surface in $\mathbb{P}^5$.  
\end{Remark}

\section{On secant defectivity of Lagrangian Grassmannians and Spinor varieties}\label{sec_def}
We recall the notions of secant varieties, secant defectivity and secant defect. We refer to \cite{Ru03} for a nice and comprehensive survey on the subject.

Let $X\subset\P^N = \mathbb{P}(V)$ be an irreducible non-degenerate variety of dimension $n$ and let
$$\Gamma_h(X)\subset X\times \dots \times X\times\G(h,V)$$
where $h\leq N$, be the closure of the graph of the rational map $\alpha: X\times  \dots \times X \dasharrow \G(h,V)$ taking $h$ general points to their linear span $\langle x_1, \dots , x_{h}\rangle$. Observe that $\Gamma_h(X)$ is irreducible and reduced of dimension $hn$. 

Let $\pi_2:\Gamma_h(X)\rightarrow\G(h,V)$ be the natural projection, and $\mathcal{S}_h(X):=\pi_2(\Gamma_h(X))\subset\G(h,V)$. Again $\mathcal{S}_h(X)$ is irreducible and reduced of dimension $hn$. Finally, consider
$$\mathcal{I}_h=\{(x,\Lambda) \: | \: x\in \Lambda\} \subset\P^N\times\G(h,V)$$
with natural projections $\pi_h$ and $\psi_h$ onto the factors. 

The {\it abstract $h$-secant variety} is the irreducible variety $\Sec_{h}(X):=(\psi_h)^{-1}(\mathcal{S}_h(X))\subset \mathcal{I}_h$. The {\it $h$-secant variety} is $\mathbb{S}ec_{h}(X):=\pi_h(Sec_{h}(X))\subset\P^N$. Then $\Sec_{h}(X)$ is an $(hn+h-1)$-dimensional variety.

The number $\delta_h(X) = \min\{hn+h-1,N\}-\dim\mathbb{S}ec_{h}(X)$ is called the \textit{$h$-secant defect} of $X$. We say that $X$ is \textit{$h$-defective} if $\delta_{h}(X) > 0$.

\stepcounter{thm}
\subsection{Osculating regularity of linear sections} 
In this section we study how the notion of osculating regularity introduced in \cite{MR19} behaves under linear sections. Let us recall \cite[Definition 5.5, Assumption 5.2]{MR19} and \cite[Definition 4.4]{AMR17}.

\begin{Definition}\label{osc_reg}
Let $X\subset\mathbb{P}^N$ be a projective variety. We say that $X$ has \textit{$m$-osculating regularity} if the following property holds: given general points $p_1,\dots,p_{m}\in X$ and an integer $s\geq 0$, 
there exists a smooth curve $C$ and morphisms $\gamma_j:C\to X$, $j=2,\dots,m$, 
such that  $\gamma_j(t_0)=p_1$, $\gamma_j(t_\infty)=p_j$, and the flat limit $T_0$ in the Grassmannian of the family of linear spaces 
$$
T_t=\left\langle T^{s}_{p_1},T^{s}_{\gamma_2(t)},\dots,T^{s}_{\gamma_{m}(t)}\right\rangle,\: t\in C\backslash \{t_0\}
$$
is contained in $T^{2s+1}_{p_1}$. We say that $\gamma_2,\dots, \gamma_m$ realize the $m$-osculating regularity of $X$ for $p_1,\dots,p_m.$

We say that $X$ has \textit{strong $2$-osculating regularity} if the following property holds: given general points $p,q\in X$ and  integers $s_1,s_2\geq 0$, there exists a smooth curve $\gamma:C\to X$ such that $\gamma(t_0)=p$, $\gamma(t_\infty)=q$ and the flat limit $T_0$ in the Grassmannian of the family of linear spaces 
$$
T_t=\left\langle T^{s_1}_p,T^{s_2}_{\gamma(t)}\right\rangle,\: t\in C\backslash \{t_0\}
$$
is contained in $T^{s_1+s_2+1}_p$.
\end{Definition}
For a discussion on the notions of $m$-osculating regularity and strong $2$-osculating regularity we refer to \cite[Section 5]{MR19} and \cite[Section 4]{AMR17}. We will need the following simple result.
\begin{Lemma}\label{lemma_simp}
Let $H\subset\mathbb{P}^n$ be a linear subspace, $H_t$ a family of linear subspaces parametrized by $\mathbb{P}^1\setminus\{0\}$, and $H_0$ its flat limit. Then 
$$\lim_{t\mapsto 0}\{H_t\cap H\}\subseteq H_0\cap H$$
\end{Lemma}
\begin{proof}
We may assume that $H = V(x_0,\ldots,x_r)\subset\mathbb{P}^n$, where $0\leq r\leq n-1$. Write
$$H_t= \left\lbrace\sum_{i=0}^n\alpha_i^1(t)x_i = \dots = \sum_{i=0}^n\alpha_i^k(t)x_i=0\right\rbrace$$

Therefore, $H_t\cap H = \{x_0 = \dots = x_r = \sum_{i=r+1}^n\alpha_i^1(t)x_i = \dots=\sum_{i=r+1}^n\alpha_i^k(t)x_i=0\}$ and 
$$\lim_{t\mapsto 0}\{H_t\cap H\}= \left\lbrace x_0=\dots = x_r = \sum_{i=r+1}^n\frac{\alpha_i^1}{t^{s_1}}(0)x_i=\dots = \sum_{i=r+1}^n\frac{\alpha_i^k}{t^{s_k}}(0)x_i=0\right\rbrace$$
where $s_j$ is the biggest power of $t$ that divides simultaneously $\alpha_{r+1}^j,\ldots,\alpha_n^j$. On the other hand
$$\lim_{t\mapsto 0}\{H_t\}\cap H = \left\lbrace x_0=\dots = x_r = \sum_{i=r+1}^n\frac{\alpha_i^1}{t^{u_1}}(0)x_i=\dots = \sum_{i=r+1}^n\frac{\alpha_i^k}{t^{u_k}}(0)x_i=0\right\rbrace$$                                         
where $u_j$ is the biggest power of $t$ that divides simultaneously $\alpha_0^j,\ldots,\alpha_n^j$. Note that $u_j\leq s_j$ for $j=1,\dots,k,$ and thus we conclude that $\lim_{t\mapsto 0}\{H_t\cap H\}\subseteq \lim_{t\mapsto 0}\{H_t\}\cap H.$
\end{proof}

As a consequence of Lemma \ref{lemma_simp} and Definition \ref{Wellbehaved} we have the following.

\begin{Proposition}\label{Prop5}
Let $X\subset\mathbb{P}^N$ be an irreducible projective variety and $Y=\mathbb{P}^k\cap X$ a linear section of $X$ that is osculating well-behaved. Assume that given general points $p_1,\ldots,p_m\in Y$ one can find smooth curves $\gamma_j:C\rightarrow X, j=2,\dots,m,$ realizing the $m$-osculating regularity of $X$ for $p_1,\ldots,p_m$ such that $\gamma_j(C)\subset Y.$  Then $Y$ has $m$-osculating regularity as well. Furthermore, the analogous statement for strong $2$-osculating regularity holds as well.
\end{Proposition}
\begin{proof}
By hypothesis given general points $p_1,\ldots,p_m\in Y$ and an integer $s\geq 0$ there exist smooth curves $\gamma_j:C\rightarrow X$ with $\gamma_j(t_0)=p_1$ and $\gamma_j(\infty)=p_j$ for $j=2,\ldots, m$ such that $\gamma_j(C)\subset Y$. Consider the family of linear spaces
$$T_t=\langle T_{p_1}^sY,T_{\gamma_2(t)}^sY,\ldots,T_{\gamma_m(t)}^sY\rangle$$
parametrized by $C\setminus\{t_0\}$. Since $Y$ is osculating well-behaved we can write $T_t$ as follows
$$\begin{array}{ccl}
T_t&=&\langle T_{p_1}^sY,T_{\gamma_2(t)}^sY,\ldots,T_{\gamma_m(t)}^sY\rangle=\langle T_{p_1}^sX\cap\mathbb{P}^s,T_{\gamma_2(t)}^sX\cap\mathbb{P}^s,\ldots,T_{\gamma_m(t)}^sX\cap\mathbb{P}^s\rangle\\
   &\subseteq&\langle T_{p_1}^sX,T_{\gamma_2(t)}^sX,\ldots,T_{\gamma_m(t)}^sX\rangle\cap\mathbb{P}^s
\end{array}$$
Therefore 
$$\lim_{t\mapsto 0}\{T_t\}\subseteq \lim_{t\mapsto 0}\{\langle T_{p_1}^sX,T_{\gamma_2(t)}^sX,\ldots,T_{\gamma_m(t)}^sX\rangle\}\cap\mathbb{P}^s = T_p^{2s+1}Y$$
where the last inclusion comes from Lemma \ref{lemma_simp}. This argument, with the obvious changes, proves that strong $2$-osculating regularity passes from $X$ to $Y$ as well.
\end{proof}

\stepcounter{thm}
\subsection{On secant defectivity of $\mathcal{LG}(n,2n)$}
In this section we will take advantage of the machinery developed in the previous sections to get a condition ensuring the non secant defectivity of $\mathcal{LG}(n,2n)\subset\mathbb{P}(V_{\omega})$. 

\begin{Proposition}\label{2OscRLG}
Let $p,q\in \mathcal{LG}(n,2n)\subset\mathbb{P}(V_{\omega})$ be general points and $s_1,s_2\geq 0$ integers. There exists a rational normal curve $\gamma:\mathbb{P}^1\rightarrow\mathcal{LG}(n,2n)$ of degree $n$ such that $\gamma(0)=p$ and $\gamma(\infty)=q$. Furthermore, consider the family of linear spaces
$$T_t=\langle T_p^{s_1}\mathcal{LG}(n,2n),T_{\gamma(t)}^{s_2}\mathcal{LG}(n,2n)\rangle$$
parametrized by $\mathbb{P}^1\setminus\{0\}$, and let $T_0$ be its flat limit in the Grassmannian. Then $T_0\subset T_p^{s_1+s_2+1}\mathcal{LG}(n,2n)$, that is $\mathcal{LG}(n,2n)\subset\mathbb{P}(V_{\omega})$ has strong $2$-osculating regularity.
\end{Proposition}
\begin{proof}
Note that taking $M_t = (I,tA)$ with $t\in\mathbb{C}$, in the parametrization of $\mathcal{LG}(n,2n)$ in (\ref{par_LG}) we see that $\mathcal{LG}(n,2n)$ is rationally connected by rational normal curves of degree $n$.

By Proposition \ref{PropOscLG} $\mathcal{LG}(n,2n) = \mathcal{G}(n,V)\cap \mathbb{P}(V_{\omega})$ is osculating well-behaved. Furthermore, since by \cite[Proposition 4.1]{MR19} and its proof $\mathcal{G}(n,V)\subset\mathbb{P}(\bigwedge^nV)$ has strong $2$-osculating regularity, with respect to the curve defined by $M_t = (I,tA)$ with $t\in\mathbb{C}$, the statement follows from Proposition \ref{Prop5}.
\end{proof}

We will need the following particular instance of \cite[Theorem 6.2]{MR19}.

\begin{thm}\label{TheoAR_2}
Let $X\subset\mathbb{P}^N$ be an irreducible projective variety having strong $2$-osculating regularity, $p\in X$ a general point, $k\geq 1$ an integer, and set $h:=\left\lfloor\frac{k+1}{2}\right\rfloor$. If $\Pi_{T_p^k}$ is generically finite then $X$ is not $(h+1)$-defective.
\end{thm}
\begin{proof}
Set $m = 2$ and $l=1$ in \cite[Theorem 6.2]{MR19}.
\end{proof}

Now, we are ready to prove the main result of this section.

\begin{thm}\label{Main_LG}
If $h\leq\left\lfloor\frac{n+1}{2}\right\rfloor$ then $\mathcal{LG}(n,2n)\subset\mathbb{P}(V_{\omega})$ is not $h$-defective.
\end{thm}
\begin{proof}
By Proposition \ref{2OscRLG} $\mathcal{LG}(n,2n)\subset\mathbb{P}(V_{\omega})$ has strong $2$-osculating regularity. Hence the statement follows from Proposition \ref{proj_osc} and Theorem \ref{TheoAR_2}.
\end{proof}

In the following table we work out the first cases of Theorem \ref{Main_LG}.  
\begin{center}
\begin{tabular}{cccl}
\hline
\textit{n} & & & Theorem \ref{Main_LG}\\
\hline
3,4 & & & not defective for $h\leq 2$\\
5,6 & & & not defective for $h\leq 3$\\
7,8 & & & not defective for $h\leq 4$\\
$\geq 9$ & & & not defective for $h\leq \left\lfloor\frac{n+1}{2}\right\rfloor$\\
\hline
\end{tabular}
\end{center}

In particular, Theorem \ref{Main_LG} improves \cite[Theorem 1.1]{BB11} as soon as $n\geq 9$. Note that by Theorem \ref{Main_LG} the Lagrangian Grassmannian $\mathcal{LG}(4,8)\subset\mathbb{P}^{42}$ is not $2$-defective and by \cite[Theorem 1.1]{BB11} it is $h$-defective for $h=3,4$.

\stepcounter{thm}
\subsection{On secant defectivity of $\mathcal{S}_n$ in the Pl\"ucker embedding}
In this section we will study the secant defectivity of $\mathcal{S}_n\subset\mathbb{P}(\bigwedge^nV_{+})$. 

\begin{Proposition}\label{2OscRSP}
Let $p,q\in \mathcal{S}_n\subset\mathbb{P}(\bigwedge^nV_{+})$ be general points, $s_1,s_2\geq 0$ integers. There exists a rational normal curve $\gamma:\mathbb{P}^1\rightarrow\mathcal{S}_n$ of degree $n$ such that $\gamma(0)=p$ and $\gamma(\infty)=q$. Furthermore, consider the family of linear spaces
$$T_t=\langle T_p^{s_1}\mathcal{S}_n,T_{\gamma(t)}^{s_2}\mathcal{S}_n\rangle$$
parametrized by $\mathbb{P}^1\setminus\{0\}$, and let $T_0$ be its flat limit in the Grassmannian. Then $T_0\subset T_p^{s_1+s_2+1}\mathcal{S}_n$, that is $\mathcal{S}_n\subset\mathbb{P}(\bigwedge^nV_{+})$ has strong $2$-osculating regularity.
\end{Proposition}
\begin{proof}
Note that taking $M_t = (I,tA)$ with $t\in\mathbb{C}$, in the parametrization of $\mathcal{S}_n$ in (\ref{parSplu}) we see that $\mathcal{S}_n$ is rationally connected by rational normal curves of degree $n$.

By Proposition \ref{Osc_S_Pl} $\mathcal{S}_n = \mathcal{G}(n,V)\cap\mathbb{P}(\bigwedge^nV_{+})$ is osculating well-behaved. Furthermore, since by \cite[Proposition 4.1]{MR19} and its proof $\mathcal{G}(n,V)\subset\mathbb{P}(\bigwedge^nV)$ has strong $2$-osculating regularity, with respect to the curve defined by $M_t = (I,tA)$ with $t\in\mathbb{C}$, the statement follows from Proposition \ref{Prop5}.
\end{proof}

The following result gives a condition ensuring the non secant defectivity of $\mathcal{S}_n\subset\mathbb{P}(\bigwedge^nV_{+})$.

\begin{thm}\label{main_SP}
If $h\leq\left\lfloor\frac{n}{2}\right\rfloor$, then $\mathcal{S}_n\subset\mathbb{P}(\bigwedge^nV_{+})$ is not $h$-defective.
\end{thm}
\begin{proof}
Since by Proposition \ref{2OscRSP} the variety $\mathcal{S}_n\subset\mathbb{P}(\bigwedge^nV_{+})$ satisfies the hypothesis of Theorem \ref{TheoAR_2} the statement follows from Proposition \ref{proj_osc}. 
\end{proof}

\begin{Remark}\label{rem2}
For instance, Theorem \ref{main_SP} yields that $\mathcal{S}_3\subset\mathbb{P}^9$ is not $1$-defective. On the other hand, by Remark \ref{rem1} it is $2$-defective.

Now, consider $\mathcal{S}_4\subset\mathbb{P}^{34}$. First, note that by (\ref{embS2loc}) in the Spinor embedding $\mathcal{S}_4$ is the quadric given by 
$$\mathcal{S}_4 = \{Z_0Z_7-Z_1Z_6+Z_2Z_5-Z_3Z_4 = 0\}\subset\mathbb{P}^{7}$$
Therefore, for $p_1,p_2,p_2\in\mathcal{S}_4\subset\mathbb{P}^{7}$ general points there is a smooth conic $C_{p_1,p_2,p_3}$ contained in $\mathcal{S}_4\subset\mathbb{P}^{7}$, namely the intersection of $\mathcal{S}_4\subset\mathbb{P}^{7}$ with the plane $\left\langle p_1,p_2,p_3\right\rangle$. As noticed in Section \ref{IsG} the Pl\"ucker embedding $\mathcal{S}_4\rightarrow\mathbb{P}^{34}$ is obtained by composing the Spinor embedding with the double Veronese embedding $\nu_2^7:\mathbb{P}^7\rightarrow\mathcal{V}_2^7\subset\mathbb{P}^{35}$. Therefore, $\mathcal{S}_{4}\subset\mathbb{P}^{34}$ is a smooth hyperplane section of the Veronese variety $\mathcal{V}_2^7\subset\mathbb{P}^{35}$. Hence $\mathcal{S}_{4}\subset\mathbb{P}^{34}$ is a smooth $6$-fold of degree $128$, and through three general points $q_1 = \nu_2^7(p_1),q_2 = \nu_2^7(p_2),q_3 = \nu_2^7(p_3)\in \mathcal{S}_4$ there is a smooth rational normal curve $\Gamma_{q_1,q_2,q_3} = \nu_2^7(C_{p_1,p_2,p_3})$ of degree $4$. Summing-up given a general point $p\in\mathbb{S}ec_3(\mathcal{S}_4)\subset\mathbb{P}^{34}$ the exists a degree $4$ rational normal curve $\Gamma$ contained in $\mathcal{S}_4$ and such that $p\in\left\langle\Gamma\right\rangle\cong\mathbb{P}^4$. Since there is a pencil of planes in $\left\langle\Gamma\right\rangle$ that are $3$-secant to $\Gamma$ and thus to $\mathcal{S}_4$ we conclude that $\mathcal{S}_4\subset\mathbb{P}^{34}$ is $3$-defective, and thus $4$-defective as well. Note that by Theorem \ref{main_SP} $\mathcal{S}_4\subset\mathbb{P}^{34}$ is not $2$-defective.

Finally, note that since given three general $3$-planes in $\mathbb{P}^7$ there exists a smooth quadric hypersurface containing them, and hence there is a Spinor variety $\mathcal{S}_4\subset\mathbb{P}^{34}$ through three general points of the Grassmannian $\mathcal{G}(4,V)\subset\mathbb{P}^{69}$, the argument above also shows the well-known $3$-secant defectivity and $4$-secant defectivity of the Grassmannian $\mathcal{G}(4,V)\subset\mathbb{P}^{69}$.   
\end{Remark}

\stepcounter{thm}
\subsection{On secant defectivity of $\mathcal{S}_n$ in the Spinor embedding}
Finally, we study the secant defectivity of $\mathcal{S}_n\subset\mathbb{P}(\Delta)$. In this case, the osculating properties of $\mathcal{S}_n$ can not be deduced from those of $\mathcal{G}(n,V)$, hence we will need a different approach.

Consider the lexicographic order on the set $\Gamma$ in (\ref{Gamma}), and let $\varnothing,J_0\in \Gamma$ be, respectively, the minimal and maximal elements of $\Gamma$. Moreover, consider the points $e_{\varnothing}=[1:0:\cdots:0], e_{J_0}=[0:\cdots:0:1]\in\mathcal{S}_n$.

A general point in a neighborhood of $e_{\varnothing}$ can be represented as a matrix $(I,A)$, where $A$ is a skew-symmetric matrix, and we can consider the following rational normal curve 
\stepcounter{thm}
\begin{equation}\label{Curve2}
\begin{array}{ccccl}
\gamma&:&\mathbb{P}^1&\longrightarrow &\mathcal{S}_n\\
&&[s:t] & \mapsto & \alpha(sI_n,tA)
\end{array}
\end{equation}
where $\alpha$ is the parametrization in (\ref{parSS}). Note that the image of $\gamma$ is a rational normal curve of degree $\lfloor\frac{n}{2}\rfloor$ such that $\gamma((1:0))=e_{\varnothing}$ and $\gamma((0:1))=e_{J_0}$.

Now, consider the following subset of $\Gamma$
$$\Lambda=\left\{\begin{array}{ll}
\{\emptyset\}\cup\Bigl\{ \{2\lambda_i-1,2\lambda_i\} \text{ where }\lambda_i\in\{1,\ldots,\frac{n}{2}\}\Bigl\}\subset\Gamma & \text{ if } n \text{ is even}\\
\{\emptyset\}\cup\Bigl\{  \{2\lambda_i,2\lambda_i+1\} \text{ where }\lambda_i\in\{1,\ldots,\frac{n-1}{2}\}\Bigl\}\subset\Gamma & \text{ if } n \text{ is odd}
\end{array}\right.$$
and for $I,J\in \Gamma$ define
\stepcounter{thm}
\begin{equation}\label{fun_sets}
\begin{array}{l}
\Gamma_k:=\{J\in \Gamma\text{ : } |J|\leq 2k\}\\ 
\Gamma^{+}_J:=\{I\in \Gamma\text{ : }J\subset I\text{ and } I\setminus J\in \Lambda\}\\ 
\Gamma^{-}_I:=\{J\in \Gamma\text{ : } I\in \Gamma^{+}_J\}=\{J\in \Gamma\text{ : }J\subset I\text{ and } I\setminus J\in \Lambda\}
\end{array} 
\end{equation}
Then we may write
$$T_{\gamma(1:t)}^s\mathcal{S}_n=T_{p_t}^s\mathcal{S}_n=
\left\langle e_I^t; |I|\leq 2s \right\rangle=
\left\langle \sum_{J\in\Gamma_I^+}
\pm t^{\frac{|J|-|I|}{2}}e_J
; |I|\leq 2s \right\rangle$$

\begin{Proposition}\label{2OscRSS}
Let $p,q\in \mathcal{S}_n\subset\mathbb{P}(\Delta)$ be general points, and $s_1,s_2\geq 0$ integers. There exists a rational normal curve $\gamma:\mathbb{P}^1\rightarrow\mathcal{S}_n$ of degree $\lfloor\frac{n}{2}\rfloor$ such that $\gamma(0)=p$ and $\gamma(\infty)=q$. Furthermore, consider the family of linear spaces
$$T_t=\langle T_p^{s_1}\mathcal{S}_n,T_{\gamma(t)}^{s_2}\mathcal{S}_n\rangle$$
parametrized by $\mathbb{P}^1\setminus\{0\}$, and let $T_0$ be its flat limit in the Grassmannian. Then $T_0\subset T_p^{s_1+s_2+1}\mathcal{S}_n$, that is $\mathcal{S}_n\subset\mathbb{P}(\Delta)$ has strong $2$-osculating regularity.
\end{Proposition}
\begin{proof}
The existence of the rational normal curve $\gamma$ has been shown in (\ref{Curve2}). Let us work on the affine chart $s=1$. All along the proof we will write $T_p^{s}$ for $T_p^{s}\mathcal{S}_n$. Proposition \ref{osc_S2} yields
$$T_{e_{\varnothing}}^{s_1}=\bigl\langle e_I,\: |I|\leq 2s_1\bigl\rangle,\quad T_{\gamma(1:t)}^{s_2}=T_{p_t}^{s_2}=
\left\langle \sum_{J\in\Gamma_I^+}\pm t^{\frac{|J|-|I|}{2}}e_J,\: |I|\leq 2s_2 \right\rangle$$
Therefore
$$T_t=\left\langle e_I\:|\:|I|\leq 2s_1\:;\:\sum_{J\in\Gamma_I^+}\pm t^{\frac{|J|-|I|}{2}}e_J, \:|I|\leq 2s_2 \right\rangle,\:t\neq0$$
We want to show that 
$$T_0\subseteq T_{e_{\varnothing}}^{s_1+s_2+1} = \bigl\langle e_I,\: |I|\leq 2s_1+2s_2+2\bigl\rangle = \bigl\{P_I=0,\:|I|>2s_1+2s_2+2\bigl\}$$ 

In order to do this it is enough to produce, for each $I\in \Gamma$ with $|I|>2s_1+2s_2+2$, a hyperplane of type
$$P_I+t\left(\sum_{J\in\Gamma; J\neq I}f(t)_{J,I}P_J\right)=0$$
where $f(t)_{J,I}\in \C[t]$ are polynomials. Recall the sets in (\ref{fun_sets}). For each $k>0$ and $I\in \Gamma$ define
$$\alpha_I^k:=|\{I\in \Lambda,\: |I|=2k \text{ and } I\subset J\}|$$
Then we have $|\Gamma^{-}_I|=\sum_{k=1}^{\alpha_I^1}\alpha_I^k$. Alternatively, note that for each $k>0$ we can obtain $\alpha_I^k$ as a function of $a_I^1$ given by $\alpha_I^k=\binom{\alpha_I^1}{k}$. Therefore, we have
$$|\Gamma^{-}_I|=\sum_{k=1}^{\alpha_I^1}\binom{\alpha_I^1}{k} = 2^{\alpha_I^1}$$
In the following, we will denote $\alpha_I^1$ simply by $\alpha_I$. Now, set
$$\Omega=\Gamma_{s_1}\bigcup\left(\bigcup_{J\in\Gamma_{s_2}}\Gamma^{+}_J\right)\subset \Gamma$$

Consider $I\in \Gamma$ such that $|I|>2s_1+2s_2+2$. If $I\notin \Omega$ then $T_t\subset\{P_I=0\}$ for any $t\neq 0$. Now, assume that $I\in \Omega$. Note that for any $e_K^t$ with non-zero homogeneous coordinate $P_I$ we have $I\in \Gamma_K^{+}$ that is $K\in \Gamma_I^{-}$. Thus, it is enough to find a hyperplane $H_I$ of type
$$
F_I=\sum_{J\in \Gamma_I^{-}}t^{\frac{|I|-|J|}{2}}c_JP_J=0
$$
with $c_I\neq 0$ and $T_t\subset H_I$ for any $t\neq 0$. In fact, we can then divide the equation by $c_I$, and we get a hyperplane $H_I$ of type
$$P_I+\frac{t}{c_I}\left(\sum_{J\in \Gamma_I^{-},J\neq I}t^{\frac{|I|-|J|}{2}-1}c_JP_J=0\right)$$
Now, we want to understand what conditions we get by requiring $T_t\subset\{F_I=0\}$ for $t\neq0$. Given $K\in \Gamma_{s_2}$ we have 
$$
\begin{array}{ccccc}
F_I(e_K^t)&=&F_I\left(\sum_{J\in\Gamma_K^+}\left(t^{\frac{|J|-|K|}{2}}e_J\right)\right)&=&\sum_{J\in\Gamma_I^{-}\cap\Gamma_K^{+}}t^{\frac{|I|-|J|}{2}}c_J\left(t^{\frac{|J|-|K|}{2}}\right)\\
&=&t^{\frac{|I|-|K|}{2}}\left(\sum_{J\in\Gamma_I^{-}\cap\Gamma_K^{+}}c_J\right)& &
\end{array}
$$
Thus
$$F_I(e_K^t)=0\: \forall\: t\neq 0\Leftrightarrow\sum_{J\in\Gamma_I^{-}\cap\Gamma_K^{+}}c_J=0$$
This is a linear condition on the coefficients $c_J$, with $J\in \Gamma_I^{-}$. Therefore
\stepcounter{thm}
\begin{equation}\label{Equa}
\begin{array}{ccl}
T_t\subset \{F_I=0\}\text{ for }t\neq0&\Leftrightarrow&\left\{\begin{array}{ll}
                                                          F_I(e_K)=0&\forall K\in\Gamma_I^{-}\cap\Gamma_{s_1}\\
                               F_I(e_K^t)=0 \: \forall\: t\neq0&\forall K\in\Gamma_I^{-}\cap\Gamma_{s_2}
                                        \end{array}\right.\\
&\Leftrightarrow&\left\{\begin{array}{ll}
c_K=0&\forall K\in\Gamma_I^{-}\cap\Gamma_{s_1}\\
\sum_{J\in\Gamma_I^{-}\cap\Gamma_K^{+}}c_J=0 \: \forall\: t\neq 0&\forall K\in\Gamma_I^{-}\cap\Gamma_{s_2}
\end{array}\right.
\end{array}
\end{equation}
The number of conditions on the $c_J$'s, $J\in \Gamma^{-}_I$ is then $|\Gamma_I^{-}\cap\Gamma_{s_1}|+|\Gamma_I^{-}\cap\Gamma_{s_2}|$ and our problem is now reduced to find a solution of the linear system given by (\ref{Equa}) in the $2^{\alpha _1}$ variables $c_J$ such that $c_I\neq 0$. Therefore, it is enough to find $\alpha_I+1$ complex numbers $c_I=c_0\neq 0,c_1,\ldots,c_{\alpha_I}$ satisfying the following equations
\stepcounter{thm}
\begin{equation}\label{Equa1}
\left\{\begin{array}{ll}
c_j=0&\forall j=\alpha_I,\ldots,\frac{|I|}{2}-s_1\\
\sum_{l=0}^{q}|\Gamma_I^{-}\cap\Gamma_{K,l}^{+}|c_{q-l}&\forall K\in\Gamma_I^{-}\cap\Gamma_{s_2}
\end{array}\right.
\end{equation}
where $\Gamma_{K,l}^{+}=\{J\in \Gamma_K^{+},\: |J|=|K|+2l\}$ and $q=\frac{|I|-|K|}{2}$.

Note that, since $K\in \Gamma_I^{-}\cap\Gamma_{s_2}$ we must have $\frac{|I|}{2}-s_2\leq\frac{|I|-|K|}{2}\leq\alpha_I$ for any $I\setminus K\in\Lambda$ with $|K|\leq 2s_2$. Therefore, (\ref{Equa1}) can be written as
$$
\left\{\begin{array}{ll}
c_j=0&\forall\: j=\alpha_I,\ldots,d_1\\
\sum_{l=0}^{j}\left(\begin{matrix}
j\\
j-l
\end{matrix}\right)c_l=0&\forall\: j=\alpha_I,\ldots,d_2
\end{array}\right.
$$
where $d_1=\frac{|I|}{2}-s_1$ and $d_2=\frac{|I|}{2}-s_2$, that is
\stepcounter{thm}
\begin{equation}\label{Equa3}
\begin{array}{ll}
\left\{\begin{array}{l}
c_{\alpha_I}=0\\
\vdots\\
c_{d_1}=0
\end{array}\right. & \left\{\begin{array}{l}
\binom{\alpha_I}{0}c_{\alpha_I}+\binom{\alpha_I}{1}c_{\alpha_I-1}+\dots+\binom{\alpha_I}{\alpha_I-1}c_1+\binom{\alpha_I}{\alpha_I}c_0=0\\
\vdots\\
\binom{d_2}{0}c_{d_2}+\binom{d_2}{1}c_{d_2-1}+\dots+\binom{d_2}{d_2-1}c_1+\binom{d_2}{d_2}c_0=0
\end{array}\right.
\end{array}
\end{equation}

Now, it is enough to show that the linear system (\ref{Equa3}) admits a solution with $c_0\neq 0$. Since $c_{\alpha_I}=\dots=c_{d_1}=0$ the system (\ref{Equa3}) can be rewritten as follows
$$
\left\{\begin{array}{l}
\binom{\alpha_I}{\alpha_I-(d_1-1)} c_{d_1-1}+\binom{\alpha_I}{\alpha_I-(d_1-2)}c_{d_1-2}+\dots+\binom{\alpha_I}{\alpha_I-1}c_1+\binom{\alpha_I}{\alpha_I}c_0=0\\
\vdots\\
\binom{d_2}{d_2-(d_1-1)}c_{d_1-1}+\binom{d_2}{d_2-(d_1-2)}c_{d_1-2}+\dots+\binom{d_2}{d_2-1}c_1+\binom{d_2}{d_2}c_0=0
\end{array}\right.
$$
Thus, it is enough to check that the $(\alpha_I-d_2+1)\times(d_1-1)$ matrix
$$
M=\left(\begin{matrix}
\binom{\alpha_I}{\alpha_I-(d_1-1)} & \binom{\alpha_I}{\alpha_I-(d_1-2)} & \dots & \binom{\alpha_I}{\alpha_I-1}\\
\vdots &\vdots & \ddots & \vdots\\
\binom{d_2}{d_2-(d_1-1)} &\binom{d_2}{d_2-(d_1-2)} & \dots & \binom{d_2}{d_2-1}
\end{matrix}\right)
$$
has maximal rank. Now, note that $\alpha_I\leq\frac{|I|}{2}$ and $\frac{|I|}{2}>s_1+s_2+1$ yield $\alpha_I-\frac{|I|}{2}+s_2+1<d_1\leq d_1-1$. Therefore, it is enough to show that the $(\alpha_I-d_2+1)\times(\alpha_I-d_2+1)$ submatrix
$$
\begin{array}{ccl}
M'&=& \left(\begin{matrix}
\binom{\alpha_I}{\alpha_I-d_2+1} & \binom{\alpha_I}{\alpha_I-d_2}& \dots & \binom{\alpha_I}{1}\\
\vdots&\vdots& \ddots &\vdots\\
\binom{d_2}{\alpha_I-d_2+1} & \binom{d_2}{\alpha_I-d_2} & \dots & \binom{d_2}{1}
\end{matrix}\right)
\end{array}$$
has non-zero determinant. Finally, since $d_2=\frac{|I|}{2}-s_2>s_1+1\geq 1$ \cite[Corollary 2]{GV85} yields that $\det(M')\neq 0$.
\end{proof}

We are ready to prove our main result on non secant defectivity of $\mathcal{S}_n\subset\mathbb{P}(\Delta)$.

\begin{thm}\label{th_main_SS}
If $h\leq\left\lfloor\frac{n+2}{4}\right\rfloor$ then $\mathcal{S}_n\subset\mathbb{P}(\Delta)$ is not $h$-defective.
\end{thm}
\begin{proof}
By Proposition \ref{2OscRSS} $\mathcal{S}_n\subset\mathbb{P}(\Delta)$ has strong $2$-osculating regularity. Therefore, to conclude it is enough to apply Proposition \ref{proj_osc} and Theorem \ref{TheoAR_2}.
\end{proof}

In the following table we work out the first cases of Theorem \ref{th_main_SS}.
\begin{center}
\begin{tabular}{cccl}
\hline
\textit{n} & & & Theorem \ref{th_main_SS}\\
\hline
6,7,8,9 & & & not defective for $h\leq 2$\\
10,11,12,13 & & & not defective for $h\leq 3$\\
$\geq 14$ & & & not defective for $h\leq\left\lfloor\frac{n+2}{4}\right\rfloor$\\
\hline
\end{tabular}
\end{center}
In particular, Theorem \ref{th_main_SS} improves the main result of \cite{An11} for $n\geq 14$. Note that Theorem \ref{th_main_SS} yields that $\mathcal{S}_7\subset\mathbb{P}^{63}$ is not $2$-defective while by \cite{An11} it is $3$-defective. Furthermore, by Theorem \ref{th_main_SS} the Spinor variety $\mathcal{S}_8\subset\mathbb{P}^{127}$ also is not $2$-defective and by \cite{An11} it is $h$-defective for $h = 3,4$.

\bibliographystyle{amsalpha}
\bibliography{Biblio}
\end{document}